\documentclass[12pt,draft]{article}
\usepackage{graphicx}
\usepackage{amsmath}
\usepackage{amsthm}
\usepackage{amsfonts}
\usepackage{amssymb, amscd}

\usepackage{authblk}

\usepackage{color}
\usepackage[all]{xy}

\usepackage{enumitem}

\usepackage{cite}

\allowdisplaybreaks

\newtheorem{thm}{Theorem}[section]
\newtheorem{cor}[thm]{Corollary}
\newtheorem{lem}[thm]{Lemma}
\newtheorem{proposition}[thm]{Proposition}
\newtheorem{example}[thm]{Example}
\theoremstyle{definition}
\newtheorem{definition}[thm]{Definition}
\theoremstyle{remark}
\newtheorem{remark}[thm]{Remark}

\newcommand{\dl}{\displaystyle}

\makeatletter
\newcommand{\subjclass}[2][1991]{%
  \let\@oldtitle\@title%
  \gdef\@title{\@oldtitle\footnotetext{#1 \emph{Mathematics subject classification}: #2}}%
}
\newcommand{\keywords}[1]{%
  \let\@@oldtitle\@title%
  \gdef\@title{\@@oldtitle\footnotetext{\emph{Keywords}: #1.}}%
}
\makeatother

\title{Generalized Derivations and Rota-Baxter Operators of $n$-ary Hom-Nambu Superalgebras}%

\author{Sami Mabrouk$^1$, Othmen Ncib$^1$, Sergei Silvestrov$^2$  \\
\footnotesize{$^1$Faculty of Sciences, University of Gafsa, BP 2100, Gafsa, Tunisia \authorcr
mabrouksami00@yahoo.fr, othmenncib@yahoo.fr   \authorcr
$^2$Division of Applied Mathematics, School of Education, Culture and Communication, \authorcr
M\"{a}lardalen University,
Box 883, 72123 V\"{a}ster{\aa}s, Sweden \authorcr
sergei.silvestrov@mdh.se}}

\subjclass[2010]{17A30,17A36,17A40,17A42}
\keywords{Hom-Lie superalgebras, $n$-ary Nambu superalgebras, $n$-ary Hom-Nambu superalgebras, $n$-Hom-Lie superalgebras, derivations, quasiderivations, Rota-Baxter operators, $3$-Hom-pre-Lie algebras}

\date{}

\begin{document}

\maketitle

\begin{abstract}
The aim of this paper is to generalise the construction of $n$-ary Hom-Lie bracket by means of an $(n-2)$-cochain of given Hom-Lie algebra to super case inducing a $n$-Hom-Lie superalgebras. We study the notion of generalized derivation and Rota-Baxter operators of $n$-ary Hom-Nambu and $n$-Hom-Lie superalgebras and their relation with generalized derivation and Rota-Baxter operators  of Hom-Lie superalgebras. We also introduce the notion of $3$-Hom-pre-Lie superalgebras which is the generalization of $3$-Hom-pre-Lie algebras.
\end{abstract}


\section*{Introduction}

Hom-Lie algebras and more general quasi-Hom-Lie algebras were introduced first by Hartwig, Larsson and Silvestrov in \cite{HartLarsSilv20032006:DefLiealgsigmaderiv}, where
the general quasi-deformations and discretizations of Lie algebras of vector fields using more general $\sigma$-derivations (twisted derivations) and a general method for construction of deformations of Witt and Virasoro type algebras based on twisted derivations have been developed.
The general quasi-Lie algebras and the subclasses of quasi-Hom-Lie algebras and Hom-Lie algebras and their more general color Hom-algebra counterparts as well as corresponding general quasi-Leibniz algebras, and thus also Hom-Leibniz algebras in context of Hom-Lie algebras, have been introduced in \cite{HartLarsSilv20032006:DefLiealgsigmaderiv,LarssonSilvJA2005:QuasiHomLieCentExt2cocyid,LarssonSilv2005:QuasiLieAlg,LarssonSilv:GradedquasiLiealg,SigSilv2006:gradedquasiLiealgWitt}.
In \cite{MakhSilv:homstructure}, the Hom-associative algebras have been introduced and shown to be Hom-Lie admissible, in the usual sense of leading to Hom-Lie algebras using commutator map as new product, thus constituting a natural generalization of associative algebras known to be Lie admissible algebras in the same sense of yielding Lie algebras using the commutator product. Moreover, in \cite{MakhSilv:homstructure}, more general $G$-Hom-associative algebras including Hom-associative algebras, Hom-Vinberg algebras (Hom-left symmetric algebras), Hom-pre-Lie algebras (Hom-right symmetric algebras) and some other Hom-algebra structures, were introduced and shown to be Hom-Lie admissible. Also, flexible Hom-algebras have been introduced and some connections to some Hom-algebra generalizations of derivations and of adjoint maps have been noticed, and the variety of $n$-dimensional Hom-Lie algebras have been considered and some classes of low-dimensional Hom-Lie algebras have been described. Since the pioneering works \cite{HartLarsSilv20032006:DefLiealgsigmaderiv,LarssonSilvJA2005:QuasiHomLieCentExt2cocyid,LarssonSilv2005:QuasiLieAlg,LarssonSilv:GradedquasiLiealg,LarssonSilv:QuasidefSl2,LarssonSigSilvJGLTA2008:QuasiLiedefFttN,MakhSilv:homstructure,RichardSilvJA2008:quasiLiesigderCtpm1,SigSilv:GLTbdSpringer2009}, Hom-algebra structures have developed in a popular broad area with increasing number of publications in various directions.
In Hom-algebra structures, defining algebra identities are twisted by linear maps. Hom-algebra structures of a given type include their classical counterparts and open more possibilities for deformations, Hom-algebra extensions of cohomological structures and representations, formal deformations of Hom-associative and Hom-Lie algebras,
Hom-Lie admissible Hom-coalgebras, Hom-coalgebras, Hom-Hopf algebras \cite{AmmarEjbehiMakhlouf:homdeformation,BenMakh:Hombiliform,LarssonSilvJA2005:QuasiHomLieCentExt2cocyid,MakhSil:HomHopf,MakhSilv:HomAlgHomCoalg,MakhSilv:HomDeform,Sheng:homrep,Yau:HomHom,Yau:HomEnv}.

The $n$-Lie algebras found their applications in many fields of Mathematics and Physics.
Ternary Lie algebras appeared first in Nambu generalization of Hamiltonian
mechanics \cite{Nambu:GenHD} using ternary bracket generalization of Poisson algebras.
The algebraic foundations of Nambu mechanics and foundations of
the theory of Nambu-Poisson manifolds have been developed in the works of Takhtajan and Daletskii in  \cite{DalTakh:leibnizLiealgNambualg,Takhtajan:foundgenNambuMech,Takhtajan:cohomology}. Filippov, in \cite{Filippov:nLie} introduced $n$-Lie algebras. Further properties, classification, and connections to other structures such as bialgebras, Yang-Baxter equation and Manin triples for $3$-Lie algebras of $n$-ary algebras were studied in
\cite{BaiRBaiCWang:realizations3Liealg,BaiGuoSheng:BialgYangBaxtereqManintri3Liealg,BaiWuLiZhou:Constrnplus1nLiealg,BaiRSongZhang:nLieclas,BaiRWangXiaoAn:nLieclaschar2,BaiRAnLi:CentStrnLiealg,BaiRChenmeng:FrattinisubalgnLiealg,BaiRMengD:CentrextnLiealg,BaiRMengD:CentrnLiealg,BaiRZhangLiShi:Innerderivalgnplus1dimnLiealg,Kasymov:nLie}.
Hom-algebra generalization of $n$-ary algebras, such as $n$-Hom-Lie algebras and other $n$-ary Hom algebras of Lie type and associative type, were introduced in \cite{AtMaSi:GenNambuAlg}, by twisting the defining identities using a set of linear maps.
A way to generate examples of such $n$-ary Hom-algebras from $n$-ary algebras of the same type has been described.
Representations and cohomology of $n$-ary multiplicative Hom-Nambu-Lie algebras have been considered in \cite{AmmarSamiMakhloufNov2010}.
Further properties, construction methods, examples, cohomology and central extensions of $n$-ary Hom-algebras have been considered in
\cite{ArlindKitMakhSilv:trucCohom3LiealgindLiealg,ArnMakhSilv:ternHomNambuLiealgindHomLie,ArnMakhSilv:ConstrnLienaryHomNambuLie,km:nary,kms:nhominduced,Yau:HomHom,Yau:HomassnaryHomNambualg,Yau:HomNambuHomMalcev,Yau:HomNambuLie}.
These generalizations include $n$-ary Hom-algebra structures generalizing the $n$-ary algebras of Lie type including $n$-ary Nambu algebras, $n$-ary Nambu-Lie algebras and $n$-ary Lie algebras, and $n$-ary algebras of associative type including $n$-ary totally associative and $n$-ary partially associative algebras.

The construction of $(n+1)$-Lie algebras induced by $n$-Lie algebras using combination of bracket multiplication with a trace, motivated by the work of Awata et al. \cite{almy:quantnambu} on the quantization of the Nambu brackets, was generalized using the brackets of general Hom-Lie algebra or $n$-Hom-Lie algebra and trace-like linear forms satisfying some conditions depending on the linear maps defining the Hom-Lie or $n$-Hom-Lie algebras in \cite{ArnMakhSilv:ternHomNambuLiealgindHomLie,ArnMakhSilv:ConstrnLienaryHomNambuLie}. The structure of $3$-Lie algebras induced by Lie algebras, classification of $3$-Lie algebras and application to constructions of B.R.S. algebras have been considered in \cite{Abramov2018:WeilAlg3LiealgBRSalg,AbramovLatt2016:classifLowdim3Liesuperalg,Abramov2017:Super3LiealgebrasinducedsuperLiealg}.
Interesting constructions of ternary Lie superalgebras in connection to superspace extension of Nambu-Hamilton equation is considered in \cite{AbramovLatt2020TernaryLieSuper}. In \cite{AbdelkhaderSamiOthmen}, a method was demonstrated of how to construct $n$-ary multiplications from the binary multiplication of a Hom-Lie algebra and a $(n-2)$-linear function satisfying certain compatibility conditions. Solvability and Nilpotency for $n$-Hom-Lie Algebras and
$(n+1)$-Hom-Lie Algebras Induced by $n$-Hom-Lie Algebras have been considered in \cite{kms:solvnilpnhomlie2020}.
In \cite{CasasLodayPirashvili:Leibniznalg}, Leibniz $n$-algebras have been studied.
The general cohomology theory for $n$-Lie algebras and Leibniz $n$-algebras was established in \cite{Rotkiewicz:cohomringnLiealg}.
The structure and classification of finite-dimensional $n$-Lie algebras were considered in \cite{LingW:StrnLiealgPhThes1993} and many
other authors. For more details of the theory and applications of $n$-Lie algebras, see \cite{DeAzcarragaIzquierdo:nAryalgrevappl}
and references therein.

Derivations and generalized derivations of different algebraic structures are an important subject of study in algebra and diverse areas. They appear in many fields of Mathematics and Physics. In particular, they appear in representation theory and cohomology theory among other areas. They have various applications relating algebra to geometry and allow the construction of new algebraic structures.  There are many generalizations of derivations. For example, Leibniz derivations \cite{KaygorodovPopov2014:Altalgadmderiv} and $\delta$-derivations of prime Lie and Malcev algebras \cite{Filippov1998:deltaderivLiealg,Filippov1999:deltaderivprimeLiealg,Filippov2000:deltaderivprimealternMalcevalg}.
The properties and structure of generalized derivations algebras of a Lie algebra and their subalgebras and quasi-derivation algebras were systematically studied in \cite{LegerLuks:GenDerivLiealg}, where it was proved for example that the quasi-derivation algebra of a Lie algebra can be embedded into the derivation algebra of a larger Lie algebra. Derivations and generalized derivations of $n$-ary algebras were considered in \cite{PojidaevSaraiva:Derivternmalcevalg,Williams:NilpnLiealg} and it was demonstrated substantial differences in structures and properties of derivations on Lie algebras and on $n$-ary Lie algebras for $n>2$. Generalized derivations of Lie superalgebras have been considered in \cite{ZhangZhang:GenDerLieSuperalg}.
Generalized derivations of Lie color algebras and $n$-ary (color) algebras have been studied in \cite{ChenMaNi:GenDerLieColAlg, Kaygorodov2012:deltaDerivnaryalg, Kaygorodov2011:nplus1Aryderivsimplenaryalg, Kaygorodov2014:nplus1AryderivsemisimpleFilipovalg, KaygorodovPopov2016:GeneralizedderivcolornLiealg}. Generalized derivations of Lie triple systems have been considered in \cite{ChenMaZhou:GenDerLieTripSyst}. Generalized derivations of various kinds can be viewed as a generalization of $\delta$-derivation. Quasi-Hom-Lie and Hom-Lie structures for $\sigma$-derivations and $(\sigma,\tau)$-derivations have been considered in \cite{ElchLundMakhSilv:bracktausigmaderiv,HartLarsSilv20032006:DefLiealgsigmaderiv,LarssonSilv:QuasidefSl2,RichardSilvJA2008:quasiLiesigderCtpm1,RichardSilvestrovGLTMPBSpr2009:QuasiLieHomLiesigmaderiv}.
Graded $q$-differential algebra and applications to semi-commutative Galois Extensions and Reduced Quantum Plane and $q$-connection was studied in \cite{AbramovJNMPh2006:gradedqdifalg,AbramovGLTMPBSpr2009:GradedqDiffAlgqConnect,AbramovRaknuz2016EngMathSpr:SemicomGaloisExtRedQPl}.
Generalized $N$-complexes coming from twisted derivations where considered in \cite{LarssonSilvestrovGLTMPBSpr2009:GenNComplTwistDer}.

Generalizations of derivations in connection with extensions and enveloping algebras of Hom-Lie color algebras have been considered in \cite{ArmakanSilvFarh:envelopalgcolhomLiealg,ArmakanSilvFarh:exthomLiecoloralg,BakyokoSilvestrov:MultiplicnHomLiecoloralg}. Generalized derivations of multiplicative $n$-ary Hom-$\Omega$ color algebras have been studied in \cite{BeitesKaygorodovPopov}. Derivations, $L$-modules, $L$-comodules and Hom-Lie quasi-bialgebras have been considered in \cite{Bakayoko:LmodcomodhomLiequasibialg,Bakayoko:LaplacehomLiequasibialg}. In \cite{kms:narygenBiHomLieBiHomassalgebras2020}, constructions of $n$-ary generalizations of BiHom-Lie algebras and BiHom-associative algebras have been considered. Generalized Derivations of $n$-BiHom-Lie algebras have been studied in \cite{BenAbdeljElhamdKaygorMakhl201920GenDernBiHomLiealg}. Color Hom-algebra structures associated to Rota-Baxter operators have been considered in context of Hom-dendriform color algebras in \cite{BakyokoSilvestrov:HomleftsymHomdendicolorYauTwi}. Rota-Baxter bisystems and covariant bialgebras, Rota-Baxter cosystems, coquasitriangular mixed bialgebras, coassociative Yang-Baxter pairs, coassociative Yang-Baxter equation and generalizations of Rota-Baxter systems and algebras, curved $\mathcal{O}$-operator systems and their connections with (tri)dendriform systems and pre-Lie algebras have been considered in \cite{MaMakhSil:CurvedOoperatorSyst,MaMakhSil:RotaBaxbisyscovbialg,MaMakhSil:RotaBaxCosyCoquasitriMixBial}.
Generalisations of derivations are important for Hom-Gerstenhaber algebras, Hom-Lie algebroids and Hom-Lie-Rinehart algebras and Hom-Poisson homology
\cite{MishraSilv:RevHomGerstalgHomLiealg}.

This paper is organized as follows. In Section \ref{sec:preliminaries} we review basic concepts of Hom-Lie, $n$-ary Hom-Nambu superalgebras and $n$-Hom-Lie algebras. We also recall some examples and classification of Hom-Lie superalgebras of dimension two. We recall  the definition of generalized  derivations of $n$-Hom-Lie superalgebras and $n$-ary Hom-Nambu superalgebras. In Section \ref{sec:nhomLiesuperindhomLiesuper} we provide a construction procedure of $n$-Hom-Lie
superalgebras starting from a binary bracket of a Hom-Lie superalgebra and multilinear form satisfying certain
conditions. To this end, we give the relation between generalized derivations  of Hom-Lie superalgebra and generalized derivations $n$-Hom-Lie algebras. In Section \ref{sec:naryhomnambusuperindHomLiesuper}, we provide a construction for $n$-ary Hom-Nambu algebra using Hom-Lie algebra. In Section \ref{sec:rotabaxternaryhomNambusuperalg} the notion of Rota-Baxter operators of $n$-ary Hom-Nambu superalgebras are introduced and some results obtained. Finally, we give the definition of $3$-Hom-pre-Lie superalgebras generalizing $3$-Hom-pre-Lie algebras in graded case.

\section{Preliminaries on $n$-ary Hom-Lie algebras and Hom-Lie superalgebras} \label{sec:preliminaries}

Throughout this paper, we will for simplicity of exposition assume that $\mathbb{K}$ is an algebraically closed
field of characteristic zero, even though for most of the general definitions and results in the paper this
assumption is not essential.

Let $V= V_{\overline 0}\oplus V_{\overline 1}$ be a finite-dimensional $\mathbb{Z}_2$-graded linear space.
Let $\mathcal{H}(\mathfrak {g})= V_{\overline 0} \cup V_{\overline 1}$ denote the set of homogeneous elements of $\mathfrak{g}$.
If $v \in V$ is a homogenous element, then its degree will be denoted by $|v|$,
where $|v|\in \mathbb{Z}_2$ and $\mathbb{Z}_2=\{\overline 0,\overline 1\}$. Let $End(V)$ be the $\mathbb{Z}_2$-graded linear space of
endomorphisms of a $\mathbb{Z}_2$-graded linear space $V= V_{\overline 0}\oplus V_{\overline 1}$ . The composition of
two endomorphisms $a\circ b$ determines the structure of superalgebra in $End(V)$,
and the graded binary commutator $[a, b] = a\circ b - (-1)^{|a||b|}b \circ a$ induces the
structure of Lie superalgebras in $End(V)$.
\subsection{Definitions and Notations}

\begin{definition}
A Hom-Lie superalgebra is a $\mathbb{Z}_2$-graded linear space $\mathfrak{g} = \mathfrak{g}_0\oplus \mathfrak{g}_1$ over a field $\mathbb{K}$
equipped with an even bilinear map $[\cdot,\cdot] : \mathfrak{g}\times \mathfrak{g} \to \mathfrak{g}$, (meaning that $[\mathfrak{g}_i,\mathfrak{g}_j ]
 \subset \mathfrak{g}_{i+j} ,\;\forall i, j\in \mathbb{Z}_2$) and an even linear
map $\alpha: \mathfrak{g}\to\mathfrak{g}$ (meaning that $\alpha(\mathfrak{g}_i)\subseteq \mathfrak{g}_i, \forall i\in \mathbb{Z}_2$).
\begin{eqnarray*}
[x,y] = -(-1)^{|x||y|} [y,x] &&  \text{(super-skew-symmetry)}  \\
\displaystyle\circlearrowleft_{x,y,z}(-1)^{|x||z|}[\alpha(x),[y,z]]=0 && \text{(super-Hom-Jacobi identity)}
\end{eqnarray*}
for all $x, y ,z  \in\mathcal{H}(\mathfrak g)$, where $\displaystyle\circlearrowleft_{x,y,z}$ denotes summation over the cyclic permutations of $x, y ,z$.
\end{definition}
\begin{definition}
  A Hom-Lie superalgebra $(\mathfrak{g},[\cdot,\cdot],\alpha)$ is called multiplicative if $\alpha([x,y])=[\alpha(x),\alpha(y)]$ for all $x,y\in \mathfrak{g}$.
\end{definition}
For any $x\in\mathfrak{g}$, define $ad_x\in End_{\mathbb{K}}(\mathfrak{g})$
by $ad_x(y) = [x, y]$, for any $y \in \mathfrak{g}$. Then the super-Hom-Jacobi identity can be written as
\begin{equation}\label{RepresentationAdjointe}
\text{ad}_{[x,y]}(\alpha(z))=\text{ad}_{\alpha(x)}\circ \text{ad}_y(z)- (-1)^{|x||y|}\text{ad}_{\alpha(y)}\circ \text{ad}_x(z)
\end{equation}
for all $\ x,y,z\in\mathcal{H}(\mathfrak{g})$.
\begin{remark}
An ordinary Lie superalgebra is a Hom-Lie superalgebra when $\alpha=id$.
\end{remark}
\begin{example}[\cite{AbdaouiMabroukMakhlouf}]
Let $\mathcal{A}$ be the complex superalgebra $\mathcal{A} = \mathcal{A}_0\oplus \mathcal{A}_1$ where $\mathcal{A}_0 = \mathbb{C}[t, t^{-1}]$ is the Laurent polynomials
in one variable and $\mathcal{A}_1 = \theta \mathbb{C}[t, t^{-1}]$, where $\theta$ is the Grassman variable $(\theta^2 = 0)$. We assume that $t$
and $\theta$ commute. The generators of $\mathcal{A}$ are of the form $t^n$ and $\theta t^n$ for $n \in \mathbb{Z}$.
For $q\in \mathbb{C}\backslash\{0, 1\}$ and $n\in  \mathbb{N}$, we set $\{n\} = \frac{1-q^n}{1-q}$, a $q$-number. The $q$-numbers have the following
properties $$\{n + 1\} = 1 + q\{n\} = \{n\} + q^n \text{\, and\, } \{n + m\} = \{n\} + q^n\{m\}.$$

Let $\mathfrak{A}_q$  be a superspace with basis $\{L_m,\ I_m| m\in\mathbb{Z}\}$ of parity $0$ and $\{ G_m,\ T_m| m\in\mathbb{Z}\}$ of parity $1$, where $L_m=-t^mD,\ I_m=-t^m,\ G_m=-\theta t^mD,\ T_m=-\theta t^m$ and $D$ is a $q$-derivation on $\mathcal{A}$ such that
$$D(t^m)=\{m\}t^m,\ D(\theta t^m)=\{m+1\}\theta t^m.$$
We define the bracket  $[\cdot,\cdot]_q:\mathfrak{A}_q\times\mathfrak{A}_q\to\mathfrak{A}_q$, with respect the super-skew-symmetry for $n,m\in\mathbb{Z}$  by
\begin{align}
&\label{crochet1}[L_m,L_n]_q=(\{m\}-\{n\})L_{m+n},\\
&\label{crochet2}[L_m,I_n]_q=-\{n\}I_{m+n},\\
&\label{crochet3}[L_m,G_n]_q=(\{m\}-\{n+1\})G_{m+n},\\
&\label{crochet4}[I_m,G_n]_q=\{m\}T_{m+n},\\
&\label{crochet5}[L_m,T_n]_q=-\{n+1\}T_{m+n},\\
&\label{crochet6}[I_m,I_n]_q=[I_m,T_n]_q=[T_m,G_n]_q=[T_m,T_n]_q=[G_m,G_n]_q=0.
\end{align}
Let $\alpha_q$ be an even linear map on $\mathfrak{A}_q$  defined on the generators by
\begin{eqnarray*}
\alpha_q(L_n)&=&(1+q^n)L_n,\hskip0.5cm \alpha_q(I_n)=(1+q^n)I_n,\\
\alpha_q(T_n)&=&(1+q^{n+1})G_n,~~\alpha_q(T_n)=(1+q^{n+1})T_n.
\end{eqnarray*}
The triple $(\mathfrak{A}_q, [\cdot,\cdot]_q, \alpha_q)$ is a Hom-Lie superalgebra, called
{\it $q$-deformed Heisenberg-Virasoro superalgebra of Hom-type}.
\end{example}
\begin{example} \label{Sl2HomLieExemple}
In {\rm \cite{AmmarMakhloufJA2010}},
the authors construct an example of Hom-Lie superalgebra, which is not
a Lie superalgebra starting from the orthosymplectic Lie superalgebra.
We consider in the sequel the matrix realization of this Lie superalgebra.

Let $osp(1,2)=V_0 \oplus V_1$  \ be  the Lie superalgebra where
$V_0$ is generated by:
$$ H=\left(
  \begin{array}{ccc}
  1 & 0& 0 \\
  0 &0 & 0 \\
    0 & 0 & -1\\
  \end{array}\right), \ \ X=\left(
  \begin{array}{ccc}
  0 & 0 & 1 \\
  0 & 0 & 0 \\
  0 & 0 & 0\\
  \end{array}
\right),\ \ Y=\left(
  \begin{array}{ccc}
  0 & 0& 0 \\
  0 & 0 & 0 \\
  1 & 0 & 0\\
  \end{array}
\right), $$ and $V_1$ is generated by
$$ F=\left(
  \begin{array}{ccc}
  0 & 0 & 0 \\
  1 & 0 & 0 \\
  0 & 1 & 0\\
  \end{array}
\right) , \ \ G=\left(
  \begin{array}{ccc}
  0 & 1& 0 \\
  0 & 0 & -1 \\
  0 & 0 & 0\\
  \end{array}
\right) .$$
Those of the defining relations that have nonzero elements in the right-hand side are
$$[H,X]=2X, \ [H,Y]=-2Y,\ [X,Y]=H,$$
$$[Y,G]=F,\ [X,F]=G,\ [H,F]=-F, \ [H,G]=G,$$
$$\ [G,F]=H, \ [G,G]=-2X,\ [F,F]=2Y. $$
 Let $\lambda \in \mathbb{R}^*$, we consider the linear map
$\alpha_{\lambda}: osp(1,2)\to osp(1,2)$ defined by:
\begin{eqnarray*}
\alpha_{\lambda}(X)=\lambda^2 X, \ \ \alpha_{\lambda}(Y)=\frac{1}{\lambda^2}Y, \ \ \alpha_{\lambda}(H)=H,\\
\alpha_{\lambda}(F)=\frac{1}{\lambda}F, \ \ \alpha_{\lambda}(G)=\lambda G.
\end{eqnarray*}
We provide a family of Hom-Lie superalgebras
$osp(1,2)_\lambda=(osp(1,2),[\cdot,\cdot]_{\alpha_\lambda},\alpha_\lambda)$, where the Hom-Lie superalgebra bracket
$[\cdot,\cdot]_{\alpha_\lambda}$ on the basis elements is given, for
$\lambda\neq 0$, by:
$$[H,X]_{\alpha_\lambda}=2\lambda^2 X,\ \
[H,Y]_{\alpha_\lambda}=-\frac{2}{\lambda^2}Y, \ \
[X,Y]_{\alpha_\lambda}=H,$$ $$
[Y,G]_{\alpha_\lambda}=\frac{1}{\lambda}F,\ \
[X,F]_{\alpha_\lambda}=\lambda G,\ \
[H,F]_{\alpha_\lambda}=-\frac{1}{\lambda}F,\ \
[H,G]_{\alpha_\lambda}=\lambda G,$$ $$ [G,F]_{\alpha_\lambda}=H,\ \
[G,G]_{\alpha_\lambda}=-2\lambda^2X,\ \
[F,F]_{\alpha_\lambda}= \frac{2}{\lambda^2}Y. $$
These Hom-Lie superalgebras are not Lie superalgebras for $\lambda\neq 1$.
 \end{example}
 \begin{thm}[\cite{WangZhangWei2016:claslowdimmulthomliesuperalg}]
 \label{classification2dimension}
   Every $2$-dimensional multiplicative Hom-Lie superalgebra $(\mathfrak{g}=\mathfrak{g}_0\oplus\mathfrak{g}_1,[\cdot,\cdot],\alpha)$ generated by $\{e_1,e_2\}$ is isomorphic to one of the following nonisomorphic Hom-Lie superalgebras. Each algebra is denoted by $\mathfrak{g}^k_{i,j}$, where $i$ is the dimension of $\mathfrak{g}_0$, $j$ is the dimension of
$\mathfrak{g}_1$, $k$ is the number.
\begin{enumerate}[label=\rm{\arabic*)}]
                                      \item $\mathfrak{g}^1_{0,2}$: is an abelian Hom-Lie superalgebra.
                                      \item $\mathfrak{g}^2_{1,1}$: is an abelian Hom-Lie superalgebra.
                                      \item $\mathfrak{g}^3_{1,1}$: $[e_0,e_1]=e_1$, $[e_1,e_1]=0$ and $\alpha=\left(
                                                                                                                \begin{array}{cc}
                                                                                                                  1 & 0 \\
                                                                                                                  0 & a \\
                                                                                                                \end{array}
                                                                                                              \right)
                                      $, $a\in \mathbb{K}$.
                                      \item $\mathfrak{g}^4_{1,1}$:$[e_0,e_1]=e_1$, $[e_1,e_1]=0$ and $\alpha=\left(
                                                                                                                \begin{array}{cc}
                                                                                                                  a & 0 \\
                                                                                                                  0 & 0 \\
                                                                                                                \end{array}
                                                                                                              \right)
                                      $, $a\neq0,1$.
                                      \item $\mathfrak{g}^5_{1,1}$: $[e_0,e_1]=0$, $[e_1,e_1]=e_0$ and $\alpha=\left(
                                                                                                                \begin{array}{cc}
                                                                                                                  a^2 & 0 \\
                                                                                                                  0 & a \\
                                                                                                                \end{array}
                                                                                                              \right)
                                      $, $a\neq0$.
                                    \end{enumerate}
 \end{thm}

Now, we recall the definitions  of $n$-ary Hom-Nambu superalgebras and $n$-Hom-Lie superalgebras, generalizing of $n$-ary Nambu superalgebras and $n$-Lie superalgebras (see \cite{AbdaouiMabroukMakhlouf}).
\begin{definition}
An \emph{$n$-ary Hom-Nambu} superalgebra $(\mathcal{N}, [\cdot,\dots,\cdot],  \widetilde{\alpha} )$ is a triple consisting of a linear space   $\mathcal{N}=\mathcal{N}_0\oplus\mathcal{N}_1$, an even
$n$-linear map $[\cdot ,\dots, \cdot ]:  \mathcal{N}^{ n}\to \mathcal{N}$ such that $[\mathcal{N}_{k_1} ,\dots,\mathcal{N}_{k_n}]
 \subset \mathcal{N}_{k_1+\dots +k_n}$ and a family
$\widetilde{\alpha}=(\alpha_i)_{1\leq i\leq n-1}$ of even linear maps $ \alpha_i:\ \ \mathcal{N}\to \mathcal{N}$, satisfying
  \begin{eqnarray}\label{NambuIdentity}
&& \forall (x_1,\dots, x_{n-1})\in \mathcal{H}(\mathcal{N})^{n-1}, \,\, (y_1,\dots,  y_n)\in \mathcal{H}(\mathcal{N})^{ n}: \nonumber \\
&& \big[\alpha_1(x_1),\dots.,\alpha_{n-1}(x_{n-1}),[y_1,\dots.,y_{n}]\big]= \\ \nonumber
&& \sum_{i=1}^{n}(-1)^{|X||Y|^{i-1}}\big[\alpha_1(y_1),\dots.,\alpha_{i-1}(y_{i-1}),[x_1,\dots.,x_{n-1},y_i], 
\alpha_i(y_{i+1}),\dots,\alpha_{n-1}(y_n)\big],
  \end{eqnarray}
where
$|X|=\displaystyle\sum_{k=1}^{n-1}|x_k|$ and $|Y|^{i-1}=\displaystyle\sum_{k=1}^{i-1}|y_k|.$

The identity \eqref{NambuIdentity} is called \emph{super-Hom-Nambu identity}.
  \end{definition}
Let $\widetilde{\alpha}:\mathcal{N}^{n-1}\to\mathcal{N}^{n-1}$ be even linear maps defined for all $X=(x_1,\ldots,x_{n-1})\in \mathcal{N}^{n-1}$ by
$\widetilde{\alpha}(X)=(\alpha_1(x_1),\ldots,\alpha_{n-1}(x_{n-1}))\in\mathcal{N}^{n-1}$.
For all $X=(x_1,\ldots,x_{n-1})\in \mathcal{N}^{n-1}$,  the map $\text{ad}_X:\mathcal{N}\to\mathcal{N}$ defined by
\begin{equation}\label{adjointMapNaire}
\text{ad}_X(y)=[x_{1},\dots,x_{n-1},y],\quad  \forall y\in \mathcal{N},
\end{equation}
is called adjoint map. Then the super-Hom-Nambu identity \eqref{NambuIdentity} may be written in terms of adjoint map as
\begin{eqnarray*}
\text{ad}_{\widetilde{\alpha} (X)}( [y_1,\dots,y_n])=
\sum_{i=1}^{n-1}(-1)^{|X||Y|^{i-1}} \left[\alpha_1(y_1),\dots,\alpha_{i-1}(y_{i-1}),
\text{ad}_X(y_{i}),\right. \\
\hspace{7cm} \left. \alpha_{i}(y_{i+1}) \dots,\alpha_{n-1}(y_n)\right].
\end{eqnarray*}

\begin{definition}
  An $n$-ary Hom-Nambu superalgebra $(\mathcal{N}, [\cdot,\dots,\cdot],  \widetilde{\alpha} )$ is called $n$-Hom-Lie superalgebra if the bracket
  $[\cdot,\dots,\cdot]$ is super-skewsym\-metric that is
  \begin{align}
& \forall\; 1\leq i\leq n-1: \nonumber \\
& [x_1,\dots,x_i,x_{i+1},\dots,x_n]=-(-1)^{|x_i||x_{i+1}|}[x_1,\dots,x_{i+1},x_i,\dots,x_n].
\label{SuperSkewSym}
  \end{align}
  It is equivalent to
 \begin{align}\label{SuperSkewSym1}
&\forall\;1\leq i<j\leq n: \nonumber \\
&\left[x_1,\dots,x_i,\dots,x_j,\dots,x_n \right]=-(-1)^{|X|^{j-1}_{i+1}(|x_i|+|x_j|)+|x_i||x_j|}  \left[x_1,\dots,x_j,\dots,x_i,\dots,x_n\right]
 \end{align}
 where $x_1,\dots ,x_n\in \mathcal{H}(\mathcal{N})$ and $|X|_{i}^{j}=\displaystyle\sum_{k=i}^{j}|x_k|.$
\end{definition}
\begin{remark}
When the maps $(\alpha_i)_{1\leq i\leq n-1}$ are all identity maps, one recovers the classical $n$-ary Nambu superalgebras.
\end{remark}

Let $(\mathcal{N},[\cdot,\dots,\cdot],\widetilde{\alpha})$ and
$(\mathcal{N}',[\cdot,\dots,\cdot]',\widetilde{\alpha}')$ be two $n$-ary Hom-Nambu
superalgebras  where $\widetilde{\alpha}=(\alpha_{i})_{1\leq i\leq n-1}$ and
$\widetilde{\alpha}'=(\alpha'_{i})_{1\leq i\leq n-1}$. A linear map $f:\mathcal{N}\to \mathcal{N}'$ is an
$n$-ary Hom-Nambu superalgebras \emph{morphism}  if it satisfies
\begin{eqnarray*}&f([x_{1},\dots,x_n])=[f(x_{1}),\dots,f(x_n)]',\\
&f \circ \alpha_i=\alpha'_i\circ f, \quad \forall i=1,\dots,n-1.
\end{eqnarray*}

In the sequel we deal sometimes with a particular class of $n$-ary Hom-Nambu superalgebras which we call $n$-ary multiplicative Hom-Nambu  superalgebras.

\begin{definition}
A \emph{multiplicative $n$-ary Hom-Nambu superalgebra }
(resp. \emph{ multiplicative $n$-Hom-Lie superalgebra}) is an $n$-ary Hom-Nambu superalgebra  (resp. $n$-Hom-Lie superalgebra) $(\mathcal{N}, [\cdot,\dots,\cdot],  \widetilde{ \alpha})$ with  $\widetilde{\alpha}=(\alpha_i)_{1\leq i\leq n-1}$
where  $\alpha_1= \dots =\alpha_{n-1}=\alpha$  and satisfying
\begin{equation}
\alpha([x_1,\dots,x_n])=[\alpha(x_1),\dots,\alpha(x_n)],\ \  \forall x_1,\dots,x_n\in \mathcal{N}.
\end{equation}
For simplicity, denote the $n$-ary multiplicative Hom-Nambu superalgebra as $(\mathcal{N}, [\cdot,\dots,\cdot],  \alpha)$ where $\alpha :\mathcal{N}\to \mathcal{N}$ is an even linear map. Also by misuse of language an element  $X\in \mathcal{N}^n$ refers to  $X=(x_1,\dots,x_{n})$ and  $\alpha(X)$ denotes
$(\alpha (x_1),\dots,\alpha (x_n))$.
\end{definition}
\begin{definition}
A multiplicative $n$-ary Hom-Nambu superalgebra $(\mathcal{N}, [\cdot ,\dots, \cdot],  \alpha)$ is called regular if $\alpha$ is bijective.
\end{definition}

\subsection{Derivations, Quasiderivations and generalized derivations  of multiplicative $n$-ary Hom-Nambu superalgebras}

In this section we recall the definition of derivation, quasiderivation and generalized derivation of multiplicative $n$-ary Hom-Nambu superalgebras.\\
Let $(\mathcal{N}, [\cdot,\dots,\cdot],  \alpha )$ be a  multiplicative $n$-ary Hom-Nambu superalgebra. We
denote by $\alpha^k$ the $k$-times composition of $\alpha$(i.e.  $\alpha^k=\displaystyle\underbrace{\alpha\circ \dots \circ\alpha}_{k-times}$ ).
In particular $\alpha^0=Id$ and $\alpha^1=\alpha$. If $(\mathcal{N}, [\cdot,\dots,\cdot],  \alpha )$ is a regular Hom-Lie superalgebra.

\begin{definition}
For any $k\geq1$, we call $\mathfrak{D}\in End(\mathcal{N})$ an $\alpha^k$-\emph{derivation }of the
 multiplicative $n$-ary Hom-Nambu superalgebra $(\mathcal{N}, [\cdot,\dots,\cdot],  \alpha )$ if
\begin{equation}\label{alphaKderiv1}[\mathfrak{D},\alpha]=0\ \ \textrm{i.e.}\ \ \mathfrak{D}\circ\alpha=\alpha\circ \mathfrak{D};\end{equation}
\begin{equation}\label{alphaKderiv2}
\mathfrak{D}[x_1,\dots,x_n]=\sum_{i=1}^n(-1)^{|\mathfrak{D}||X|^{i-1}}[\alpha^k(x_1),\dots,\alpha^k(x_{i-1}),\mathfrak{D}(x_i),\alpha^k(x_{i+1}),\dots,\alpha^k(x_n)].
\end{equation}
We denote by $Der_{\alpha^k}(\mathcal{N})$ the set of $\alpha^k$-derivations of
the  multiplicative $n$-Hom-Lie  superalgebra $\mathcal{N}$.
\end{definition}

For $X=(x_1,\dots,x_{n-1})\in \mathcal{N}^{ n-1}$ satisfying $\alpha(X)=X$ and $k\geq 1$,
we define the map $\text{ad}^k_X\in End(\mathcal{N})$ by
\begin{equation}\label{ad_k(u)}
\text{ad}^k_X(y)=[x_1,\dots,x_{n-1},\alpha^k(y)]\ \ \forall y\in \mathcal{N}.
\end{equation}
Then, we find the following result.

\begin{lem}
The map $\text{ad}^k_X$ is an $\alpha^{k+1}$-derivation (called inner $\alpha^{k+1}$-derivation), and $|\text{ad}^k_X|=|X|$.
\end{lem}
We denote by $Inn_{\alpha^k}(\mathcal{N})$ the space generate by all the inner $\alpha^{k+1}$-derivations.
For any $\mathfrak{D}\in Der_{\alpha^k}(\mathcal{N})$ and $\mathfrak{D}'\in Der_{\alpha^{k'}}(\mathcal{N})$ we define their supercommutator $[\mathfrak{D},\mathfrak{D}']$ as usual:
\begin{equation}\label{DerivationsCommutator}[\mathfrak{D},\mathfrak{D}']=\mathfrak{D}\circ \mathfrak{D}'
-(-1)^{|\mathfrak{D}||\mathfrak{D}'|}\mathfrak{D}'\circ \mathfrak{D},
\end{equation}
then $[\mathfrak{D},\mathfrak{D}']\in Der_{\alpha^{k+k'}}(\mathcal{N})$.
Set $Der(\mathcal{N})=\dl\bigoplus_{k\geq 0}Der_{\alpha^k}(\mathcal{N})$ and $Inn(\mathcal{N})=\dl\bigoplus_{k\geq 0}Inn_{\alpha^k}(\mathcal{N})$,
 the pair $(Der(\mathcal{N}),[\cdot,\cdot])$ is a Lie superalgebra.

\begin{definition}
  Let $(\mathcal N, [\cdot,\dots,\cdot], \alpha)$ be a multiplicative $n$-ary Hom-Nambu superalgebra. An endomorphism $\mathfrak{D}\in End(\mathcal N)$ is said to be an $\alpha^k-$quasiderivation, if
there exists  an endomorphism $\mathfrak{D}'\in End(\mathcal N)$ such that
$$\displaystyle\sum_{i=1}^n(-1)^{|\mathfrak{D}||X|^{i-1}}[\alpha^k(x_1),\dots,\mathfrak{D}(x_i),\dots, \alpha^k(x_n)]=\mathfrak{D}'([x_1,\dots, x_n]),$$
for all $x_1,\dots,x_n\in \mathcal N$.
We  call $\mathfrak{D}'$ the endomorphism associated to $\alpha^k-$quasiderivation $\mathfrak{D}$.
\end{definition}
We denote the set of $\alpha^k$-quasiderivations by $QDer_{\alpha^k}(\mathcal N)$ and $$QDer(\mathcal N)=\dl\bigoplus_{k\geq0}QDer_{\alpha^k}(\mathcal N).$$
\begin{definition} An endomorphism $\mathfrak{D}$ of a multiplicative $n$-ary Hom-Nambu superalgebra $(\mathcal N, [\cdot,\dots,\cdot], \alpha)$  is called a generalized $\alpha^k$-derivation   if there
exist linear mappings $$\mathfrak{D}',\mathfrak{D}'', \dots ,\mathfrak{D}^{(n-1)},\mathfrak{D}^{(n)} \in End(\mathcal{N}) $$ such that
\begin{equation}\label{genereliseDerivetin}
\mathfrak{D}^{(n)}([x_1, \dots , x_n])=\sum_{i=1}^n(-1)^{|\mathfrak{D}^{(i-1)}||X|^{i-1}}[\alpha^k(x_1), \dots ,\mathfrak{D}^{(i-1)}(x_i), \dots, \alpha^k(x_n)],
\end{equation}
for all $x_1,\dots , x_n\in\mathcal{N}$. An $(n + 1)$-tuple $(\mathfrak{D},\mathfrak{D}',\mathfrak{D}'', \dots ,\mathfrak{D}^{(n-1)},\mathfrak{D}^{(n)})$ is called an $(n + 1)$-ary
$\alpha^k$-derivation.
\end{definition}
We denote the set of generalized $\alpha^k$-derivations by $GDer_{\alpha^k}(\mathcal{N})$ and    $$GDer(\mathcal{N})=\dl\bigoplus_{k\geq0}GDer_{\alpha^k}(\mathcal{N}).$$


\section{$n$-Hom-Lie superalgebras induced by Hom-Lie superalgebras} \label{sec:nhomLiesuperindhomLiesuper}

In \cite{GuanChenSun}, the authors introduced a construction of a $3$-Hom-Lie superalgebra from a Hom-Lie
superalgebra. It is called $3$-Hom-Lie superalgebra induced by Hom-Lie superalgebra. In this section we generalize this construction to the $n$-ary Hom-algebras
by the approach in \cite{Abramov2017:Super3LiealgebrasinducedsuperLiealg}.

Let  $ (\mathfrak{g},[\cdot,\cdot],\alpha)$ be a multiplicative Hom-Lie superalgebra and $ \mathfrak{g}^*$ be its dual superspace. Fix an even
element of the dual space $\varphi\in \mathfrak{g}^*$. Define the triple product
as follows
\begin{align}\label{TripleProduct}
&\forall\ x,\ y,\ z\in \mathcal{H}(\mathfrak{g}): \nonumber \\
&[x,y,z]=\varphi(x)[y,z]+(-1)^{|x|(|y|+|z|)}\varphi(y)[z,x]+(-1)^{|z|(|x|+|y|)}\varphi(z)[x,y].
\end{align}
Obviously this triple product is   super-skew-symmetric. It is straightforward to compute the left-hand side and the
right-hand side of the super-Hom-Nambu identity \eqref{NambuIdentity} if $\varphi\circ\alpha=\varphi$
and
\begin{equation}\label{ConditionTrace}
\varphi(x)\varphi([y,z])+(-1)^{|x|(|y|+|z|)}\varphi(y)\varphi([z,x])+(-1)^{|z|(|x|+|y|)}\varphi(z)\varphi([x,y])=0.
\end{equation}

Now we consider $\varphi$ as an even $\mathbb{K}$-valued cochain of degree one of the Chevalley-Eilenberg
complex of a Hom-Lie superalgebra $\mathfrak g$. Let coboundary operator $\delta:\wedge^{k}\mathfrak{g}^*\to\wedge^{k}\mathfrak{g}^*$ be defined by
 \begin{equation}\label{OperatorCobord}
    \delta f(x_1,\dots,x_{k+1})=\sum_{i<j}(-1)^{i+j+1}(-1)^{\gamma^X_{ij}}f([x_i,x_j]_{\mathfrak g},\alpha(x_1)\dots,\widehat{x_i},\dots,\widehat{x_j},\dots,\alpha(x_{k+1})),
 \end{equation}
 where $\gamma^X_{ij}=|X|^n_{j+1}(|x_i|+|x_j|)+|x_i||X|_{j-1}^{i+1}$, for $f\in\wedge^{k}\mathfrak{g}^*$  and for all $ x_1,\dots,x_{k+1}\in \mathcal{H}(\mathfrak{g})$.
Then, $\delta\varphi(x, y) = \varphi([x, y])$. Finally, we can define the wedge product of two cochains $\varphi$ and $\delta\varphi$, which is the cochain of
degree three by
$$\varphi\wedge\delta\varphi(x, y, z) =\varphi(x)\varphi([y,z])+(-1)^{|x|(|y|+|z|)}\varphi(y)\varphi([z,x])+(-1)^{|z|(|x|+|y|)}\varphi(z)\varphi([x,y]).$$
Hence \eqref{ConditionTrace} is equivalent to $\varphi\wedge\delta\varphi=0$.
Thus, if an $1$-cochain $\varphi$ satisfies the equation \eqref{ConditionTrace}, then the triple product \eqref{TripleProduct} is the
$3$-Hom-Lie bracket, and  we will call this multiplicative $3$-Hom-Lie bracket
the quantum Hom-Nambu bracket induced by an even $1$-cochain.
\begin{definition}
Let $\phi\in \wedge^{n-2}\mathfrak{g}^*$ be an even $(n-1)$-cochain, we define the $n$-ary product as follows
\begin{equation}\label{nProduct}
[x_1,\dots,x_n]_\phi=\sum_{i<j}^{n}(-1)^{i+j+1}(-1)^{\gamma^X_{ij}}\phi(x_1,\dots,\hat{x_i},\dots,\hat{x_j},\dots,x_n)[x_i,x_j],
\end{equation}
for all $x_1,\dots,x_n\in\mathcal{H}(\mathfrak{g})$.
\end{definition}
It is clair that $[\cdot,\dots,\cdot ]_\phi$ is an even $n$-linear map.
\begin{proposition}
  The $n$-ary product $[\cdot,\dots,\cdot ]_\phi$ is super-skew-symmetric.
\end{proposition}
\begin{proof}
  Let $x_1,\dots,x_n\in \mathcal{H}(\mathfrak g)$ and fix an integer  $1\leq i\leq n-1$. Then,
  \begin{align*}
  &[x_1,\dots,x_i,x_{i+1},\dots,x_n]_{\phi}= \\
  &\displaystyle\sum_{k<l<i}
  (-1)^{k+l+1}(-1)^{\gamma^X_{kl}}\phi(x_1,\dots,\widehat{x}_k,\dots,\widehat{x}_l\dots,x_i,\dots,x_n)[x_k,x_l]\\
  &+\displaystyle\sum_{i+1<k<l} (-1)^{k+l+1}(-1)^{\gamma^X_{kl}}\phi(x_1,\dots,x_i,x_{i+1},\dots,\widehat{x}_k\dots,\widehat{x}_l,\dots,x_n)[x_k,x_l]\\
  &+\displaystyle\sum_{k=i<l\neq i+1} (-1)^{i+l+1}(-1)^{\gamma^X_{il}}\phi(x_1,\dots,\widehat{x}_i,x_{i+1},\dots,\widehat{x}_l,\dots,x_n)[x_i,x_l]\\ &+\displaystyle\sum_{k<l=i} (-1)^{k+i+1}(-1)^{\gamma^X_{ki}}\phi(x_1,\dots,\widehat{x}_k,\dots,\widehat{x}_i,x_{i+1},\dots,x_n)[x_k,x_i]\\
  &+\displaystyle\sum_{i\neq k<l=i+1} (-1)^{k+i}(-1)^{\gamma^X_{k,i+1}}\phi(x_1,\dots,\widehat{x}_k,\dots,x_i,\widehat{x}_{i+1},\dots,x_n)[x_k,x_{i+1}]\\
  &+\displaystyle\sum_{k=i+1<l} (-1)^{l+i}(-1)^{\gamma^X_{i+1,l}}\phi(x_1,\dots,x_i,\widehat{x}_{i+1},\dots,\widehat{x}_l,\dots,x_n)[x_{i+1},x_l]\\
  &+ (-1)^{\gamma^X_{i,i+1}}\phi(x_1,\dots,\widehat{x}_i,\widehat{x}_{i+1},\dots,x_n)[x_i,x_{i+1}]\\
  &=S_1+\dots+S_7
  \end{align*}
  and
  \begin{align*}
  &[x_1,\dots,x_{i+1},x_i,\dots,x_n]_{\phi}=\\
  &\displaystyle\sum_{k<l<i}
  (-1)^{k+l+1}(-1)^{\gamma^X_{kl}}\phi(x_1,\dots,\widehat{x}_k,\dots,\widehat{x}_l\dots,x_{i+1},x_i,\dots,x_n)[x_k,x_l]\\
  &+\displaystyle\sum_{i+1<k<l} (-1)^{k+l+1}(-1)^{\gamma^X_{kl}}\phi(x_1,\dots,x_{i+1},x_i,\dots,\widehat{x}_k\dots,\widehat{x}_l,\dots,x_n)[x_k,x_l]\\
  &+\displaystyle\sum_{k=i<l\neq i+1} (-1)^{i+l+1}(-1)^{\zeta^X_{il}}\phi(x_1,\dots,\widehat{x}_{i+1},x_i,\dots,\widehat{x}_l,\dots,x_n)[x_{i+1},x_l]\\ &+\displaystyle\sum_{k<l=i} (-1)^{k+i+1}(-1)^{\zeta^X_{ki}}\phi(x_1,\dots,\widehat{x}_k,\dots,\widehat{x}_{i+1},x_i,\dots,x_n)[x_k,x_{i+1}]\\
  &+\displaystyle\sum_{i\neq k<l=i+1} (-1)^{k+i}(-1)^{\zeta^X_{k,i+1}}\phi(x_1,\dots,\widehat{x}_k,\dots,x_{i+1},\widehat{x}_i,\dots,x_n)[x_k,x_i]\\
  &+\displaystyle\sum_{k=i+1<l} (-1)^{l+i}(-1)^{\zeta^X_{i+1,l}}\phi(x_1,\dots,x_{i+1},\widehat{x}_i,\dots,\widehat{x}_l,\dots,x_n)[x_i,x_l]\\
  &+ (-1)^{\gamma^X_{i,i+1}}\phi(x_1,\dots,\widehat{x}_i,\widehat{x}_{i+1},\dots,x_n)[x_{i+1},x_i]\\
  &=S'_1+\dots+S'_7
  \end{align*}
  where
  \begin{align*}
  \zeta^X_{il}&=|X|^n_{l+1}(|x_{i+1}|+|x_l|)+|x_{i+1}|(|x_i|+|x_{i+2}|+\dots+|x_{l-1}|)\\
  &=|X|^n_{l+1}(|x_{i+1}|+|x_l|)+|x_{i+1}|(|x_{i+2}|+\dots+|x_{l-1}|)+|x_{i+1}||x_i|\\
  &=\gamma^X_{i+1,i}+|x_{i+1}||x_i|.
  \end{align*}
So we conclude that $S'_3=-(-1)^{|x_i||x_{i+1}|}S_6$.
In the same way, it is easy to see also that
$S'_4=-(-1)^{|x_i||x_{i+1}|}S_5$, $S'_5=-(-1)^{|x_i||x_{i+1}|}S_4,$ and $S'_6=-(-1)^{|x_i||x_{i+1}|}S_3$.

The super-skew-symmetry of $\phi$ gives that $S'_1=-(-1)^{|x_i||x_{i+1}|}S_1,\;S'_2=-(-1)^{|x_i||x_{i+1}|}S_2$ and
  $S'_7=-(-1)^{|x_i||x_{i+1}|}S_7$.
  Finally we get $$[x_1,\dots,x_i,x_{i+1},\dots,x_n]_{\phi}=-(-1)^{|x_i||x_{i+1}|}[x_1,\dots,x_{i+1},x_i,\dots,x_n]_{\phi}.$$
\end{proof}
Given $X=(x_1,\dots,x_{n-3})\in \wedge^{n-3}\mathcal{H}(\mathfrak g)$, $Y=(y_1,\dots,y_{n})\in \wedge^{n}\mathcal{H}(\mathfrak g)$
and $z\in \mathcal{H}(\mathfrak g)$, we define the linear map $\phi_X$ by $\phi_X(z)=\phi(X,z),$ and
\begin{align*}
\phi\wedge\delta\phi_X(Y)&=\sum_{i<j}^{n}(-1)^{i+j}(-1)^{\gamma^Y_{ij}}\phi(y_1,\dots \hat{y_i}\dots\hat{y_j}\dots,y_{n})\delta\phi_X(y_i,y_j)\\
&=\sum_{i<j}^{n}(-1)^{i+j}(-1)^{\gamma^Y_{ij}}\phi(y_1,\dots \hat{y_i}\dots\hat{y_j}\dots,y_{n})\phi_X([y_i,y_j]).\end{align*}
\begin{thm}\label{HomLieSupAlToNaryHomLieSupAlg}
  Let $(\mathfrak g,[\cdot,\cdot],\alpha)$ be a  multiplicative Hom-Lie superalgebra, $\mathfrak g^*$  its dual and $\phi$ be
an even cochain of degree $n-2$, i.e. $\phi\in\wedge^{n-2}\mathfrak g^*$. The linear space   $\mathfrak g$
equipped with the n-ary product \eqref{nProduct} and the even linear map $\alpha$ is a multiplicative n-Hom-Lie superalgebra if and only if
\begin{align}\label{NHomLieProduct}
 & \phi\wedge\delta\phi_X=0,\ \forall X\in \wedge^{n-3}\mathcal{H}(\mathfrak{g}),\\
  & \phi\circ(\alpha\otimes Id\otimes\dots\otimes Id)=\phi.\label{NHomLieProduct1}
\end{align}
\end{thm}
\begin{proof}
Firstly, if $(x_1,\dots,x_n)\in  \wedge^{ n}\mathcal{H}(\mathfrak g)$, then
\begin{align*}
&[\alpha(x_1),\dots,\alpha(x_n)]_{\phi} =\\
&=\sum_{i<j}^{n}(-1)^{i+j+1}(-1)^{\gamma^Y_{ij}}\phi(\alpha(x_1),
\dots,\widehat{\alpha(x_i)},\dots,\widehat{\alpha(x_j)},\dots,\alpha(x_n))[\alpha(x_i),\alpha(x_j)]\\
&=\sum_{i<j}^{n}(-1)^{i+j+1}(-1)^{\gamma^X_{ij}}\phi(x_1,\dots,\hat{x_i},\dots,\hat{x_j},\dots,x_n)\alpha([x_i,x_j])\\
&=\alpha([x_1,\dots,x_n]_{\phi}).
\end{align*}
  Secondly, for $(x_1,\dots,x_{n-1})\in \wedge^{ n-1}\mathcal{H}(\mathfrak g)$ and $(y_1,\dots,y_n)\in \wedge^{ n}\mathcal{H}(\mathfrak g)$, we have
\begin{align*}
   &[\alpha(x_1),\dots,\alpha(x_{n-1}),[y_1,\dots,y_n]_{\phi}]_{\phi}\\&=\displaystyle\sum_{i<j}(-1)^{i+j+1} (-1)^{\gamma^Y_{ij}}\phi(y_1,\dots,\widehat{y}_i,\dots,\widehat{y}_j,\dots,y_n)[\alpha(x_1),\dots,\alpha(x_{n-1}),[y_i,y_j]]_{\phi}\\
   &=\displaystyle\sum_{i<j}\displaystyle\sum_{k<l\leq n-1}(-1)^{i+j+k+l}
   (-1)^{\gamma^Y_{ij}+\gamma^X_{kl}}(-1)^{(|x_k|+|x_l|)(|x_i|+|x_j|)} \\
   &\phi(\alpha(x_1),\dots,\widehat{\alpha(x_k)},\dots,\widehat{\alpha(x_l)},\dots,[y_i,y_j]) \phi(y_1,\dots,\widehat{y}_i,\dots,\widehat{y}_j,\dots,y_n)[\alpha(x_k),\alpha(x_l)]\\
   &+\displaystyle\sum_{i<j}\displaystyle\sum_{k<n}(-1)^{i+j+k}
   (-1)^{\gamma^Y_{ij}}(-1)^{|x_k||X^{k+1}|} \\
  &\phi(\alpha(x_1),\dots,\widehat{\alpha(x_k)},\dots,\alpha(x_{(n-1)}),\dots,\widehat{[y_i,y_j]}) \phi(y_1,\dots,\widehat{y}_i,\dots,\widehat{y}_j,\dots,y_n)[\alpha(x_k),[y_i,y_j]].
 \end{align*}
 The terms $[\alpha(x_k),[y_i,y_j]]$ are simplified by identity of Jacobi in the second half of the Filippov identity. Now, we group together the other terms according to their coefficient $[\alpha(x_i),\alpha(x_j)]$. For example, if we fixed $(k, l)$ and, if we collect all the terms containing the commutator $[\alpha(x_k),\alpha(x_l)]$,  then we get the expression
\begin{multline*}
\hspace{-7pt} \Big(\displaystyle\sum_{i<j}(-1)^{i+j+k+l}(-1)^{\gamma^Y_{ij}+\gamma^X_{kl}}(-1)^{(|x_k|+|x_l|)(|x_i|+|x_j|)}
   \phi(\alpha(x_1),\dots,\widehat{\alpha(x_k)},\dots,\widehat{\alpha(x_l)},\dots,[y_i,y_j]) \\
\phi(y_1,\dots,\widehat{y}_i,\dots,\widehat{y}_j,\dots,y_n)\Big)[\alpha(x_k),\alpha(x_l)].
\end{multline*}

Hence the $n$-ary product \eqref{nProduct} will satisfy the $n$-ary Filippov-Jacobi identity if for
any elements $X=(x_1,\dots,x_{n-3})\in\wedge^{n-3}\mathcal{H}(\mathfrak g)$ and $Y=(y_1,\dots,y_n)\in\wedge^n \mathcal{H}(\mathfrak g)$  we require
$$\Big(\displaystyle\sum_{i<j}^n(-1)^{i+j}(-1)^{\gamma^Y_{ij}}
   \phi(\alpha(x_1),\dots,\alpha(x_{n-3}),[y_i,y_j])
   \phi(y_1,\dots,\widehat{y}_i,\dots,\widehat{y}_j,\dots,y_n)\Big)=0.$$
\end{proof}
\begin{example}[\cite{WangZhangWei2016:claslowdimmulthomliesuperalg}] Consider a $3$-dimensionally graded linear space $L=L_0\oplus L_1$, where $L_0$ is generated by $e_1,e_2$ and   $L_1$ is generated by $e_3$. Define an even linear map $\alpha:L\to L$ by
$$\alpha(e_1)=a^2e_1,\ \ \alpha(e_2)=e_2 \ \ \alpha(e_3)=ae_2$$ and an even super-skewsymmetric bilinear map $[\cdot,\cdot]:L\times L\to L$ given by
$$[e_1,e_2]=[e_1,e_3]=0,\ \ [e_2,e_3]=e_3,\ \ [e_3,e_3]=e_1.$$
Then $(L,[\cdot,\cdot],\alpha)$ is a multiplicative Hom-Lie superalgebra. Define an even linear form $\phi:L\to \mathbb{K}$ given by $\phi(e_1)=0$ and $\phi(e_2)=b$. Then, we have
$$\phi\circ\alpha=\phi\ \ \textrm{and}\ \ \phi\wedge\delta\phi=0.$$
Therefore, using Theorem \ref{HomLieSupAlToNaryHomLieSupAlg}, we can construct a multiplicative $3$-Hom-Lie superalgebra $(L,[\cdot,\cdot,\cdot]_\phi,\alpha)$, where the ternary bracket $[\cdot,\cdot,\cdot]_\phi$ is given by $[e_2,e_3,e_3]=\phi(e_1)[e_3,e_3]=be_1$ and $a,b$ are parameters.

\end{example}
\begin{definition}
  Let $\phi:\mathfrak g\wedge\dots\wedge \mathfrak  g\to \mathbb{K}$ be an even super-skewsymmetric linear form of   the multiplicative  Hom-Lie superalgebra $(\mathfrak g,[\cdot,\cdot],\alpha)$, then $\phi$ is called supertrace if:
   $$\phi\circ([\cdot,\cdot]\wedge Id\wedge\dots\wedge Id)=0~~~~~\text{and}~~ \phi\circ(\alpha\wedge Id\wedge\dots\wedge Id)=\phi.
   ~$$
\end{definition}
\begin{cor}
  Let $\phi:\mathfrak g^{\wedge n-2}\to \mathbb{K}$ be a supertrace of  Hom-Lie superalgebra $(\mathfrak g,[\cdot,\cdot],\alpha)$, then $\mathfrak g_\phi=(\mathfrak g,[.,\dots,.]_\phi,\alpha)$ is a $n$-Hom-Lie superalgebra.
\end{cor}
%
\begin{proposition}
Let $(\mathfrak g,[\cdot,\cdot],\alpha)$ be a Hom-Lie superalgebra, and let $\mathfrak{D} \in Der(\mathfrak{g})$ be an $\alpha^k$-derivation such that
$$\sum_{i=1}^{n-2}(-1)^{|\mathfrak{D}||X|^{i-1}}\phi(x_1,\dots \mathfrak{D}(x_i),\dots,x_{n-2})=0.$$
Then $\mathfrak{D}$ is  an $\alpha^k$-derivation of the $n$-Hom-Lie superalgebra $(\mathfrak g,[\cdot,\dots,\cdot]_\phi,\alpha)$.
\end{proposition}
\begin{proof}
Let $X=(x_1,\dots,x_n)\in\wedge^n \mathcal{H}(\mathfrak g)$, on the one hand we get

\begin{align*}
\mathfrak{D}([x_1,\dots,x_n]_\phi)&= \mathfrak{D}\Big(\displaystyle\sum_{i<j}(-1)^{i+j+1}
\phi(\alpha(x_1),\dots,\alpha(\widehat{x}_i),\dots,\alpha(\widehat{x}_j),\dots,\alpha(x_n))[\alpha(x_i),\alpha(x_j)]\Big)\\
&=\displaystyle\sum_{i<j}(-1)^{i+j+1}
\phi(\alpha(x_1),\dots,\alpha(\widehat{x}_i),\dots,\alpha(\widehat{x}_j),\dots,\alpha(x_n))\mathfrak{D}([\alpha(x_i),\alpha(x_j)])\\
&=\displaystyle\sum_{i<j}(-1)^{i+j+1}
\phi(x_1,\dots,\widehat{x}_i,\dots,\widehat{x}_j,\dots,x_n)[\alpha(\mathfrak{D}(x_i)),\alpha^{k+1}(x_j)]\\
&+\displaystyle\sum_{i<j}(-1)^{i+j+1}
\phi(x_1,\dots,\widehat{x}_i,\dots,\widehat{x}_j,\dots,x_n)[\alpha^{k+1}(x_i),\alpha(\mathfrak{D}(x_j))],
\end{align*}
 on the other hand, we have
 \begin{align*}
&\displaystyle\sum_{l=1}^n[\alpha^k(x_1),\dots,\alpha^k(x_{l-1}),\mathfrak{D}(x_l),\dots,\alpha^k(x_{l+1}),\dots,\alpha^k(x_n)]_\phi \\&=
\displaystyle\sum_{l=1}^n\displaystyle\sum_{i<j\;;\;i,j\neq l}(-1)^{i+j+1}
\phi(\alpha^k(x_1),\dots,\widehat{\alpha^k(x_i)},\dots,\mathfrak{D}(x_l),\dots,\widehat{\alpha^k(x_j)},\dots,\alpha^k(x_n)) \\
& \hspace{12cm} [\alpha^k(x_i),\alpha^k(x_j)] \\
&+\displaystyle\sum_{l=1}^n\displaystyle\sum_{i<l}(-1)^{i+l+1}
\phi(\alpha^k(x_1),\dots,\widehat{\alpha^k(x_i)},\dots,\widehat{\mathfrak{D}(x_l)},\dots,\alpha^k(x_n))[\alpha^k(x_i),\mathfrak{D}(x_l)]\\
&+\displaystyle\sum_{l=1}^n\displaystyle\sum_{l=i<j}(-1)^{j+l+1}
\phi(\alpha^k(x_1),\dots,\widehat{\mathfrak{D}(x_l)},\dots,\widehat{\alpha^k(x_j)},\dots,\alpha^k(x_n))[\mathfrak{D}(x_l),\alpha^k(x_j)].
\end{align*}
If $\mathfrak{D}$ is an $\alpha^k$-derivation then $$\mathfrak{D}([x_1,\dots,x_n]_\phi)=\displaystyle\sum_{l=1}^n(-1)^{|\mathfrak{D}||X|^{i-1}}[\alpha^k(x_1),\dots,\alpha^k(x_{l-1}),\mathfrak{D}(x_l),\dots,\alpha^k(x_{l+1}),\dots,\alpha^k(x_n)]_\phi,$$ which gives
\begin{multline*}
\hspace{-10pt} \sum_{\substack{i<j\\i,j\neq l}}(-1)^{i+j+1}\Big(\displaystyle\sum_{l=1}^n(-1)^{|\mathfrak{D}||X|^{l-1}}\displaystyle
\phi(\alpha^k(x_1),\dots,\widehat{\alpha^k(x_i)},\dots,\mathfrak{D}(x_l),\dots,\widehat{\alpha^k(x_j)},\dots,\alpha^k(x_n))\Big) \\ %
[\alpha^k(x_i),\alpha^k(x_j)]=0.
\end{multline*}
Finally, if we fixed $(i,j)$ we have
$$\displaystyle\sum_{l=1}^{n-2}(-1)^{|\mathfrak{D}||X|^{l-1}}\displaystyle
\phi(\alpha^k(x_1),\dots,\mathfrak{D}(x_l),\dots,\alpha^k(x_{n-2}))=0.$$
\end{proof}
\begin{proposition}

Let $(\mathfrak g,[\cdot,\cdot],\alpha)$ be a Hom-Lie superalgebra and let $\mathfrak{D}\in QDer(\mathfrak g)$ be an $\alpha^k$-quasi-derivation and $\mathfrak{D}':\mathfrak{g}\to\mathfrak{g}$ the endomorphism associated to $\mathfrak{D}$ such that
$$\sum_{i=1}^{n-2}(-1)^{|\mathfrak{D}||X|^{i-1}}\phi(x_1,\dots \mathfrak{D}(x_i),\dots,x_{n-2})=0.$$
Then $\mathfrak{D}$ is  an $\alpha^k$-quasi-derivation of the $n$-Hom-Lie superalgebra $(\mathfrak g,[\cdot,\dots,\cdot]_\phi,\alpha)$ with the same endomorphism associated to $\mathfrak{D}'$.
\end{proposition}



\section{$n$-ary Hom-Nambu superalgebras induced by Hom-Lie superalgebras } \label{sec:naryhomnambusuperindHomLiesuper}


In this section we construct  an $n$-ary Hom-Nambu superalgebras with a help of a given Hom-Lie superalgebra by analogue of Hom-Lie triple system given in \cite{Yau:HomNambuLie} in graded case.
 Let $(\mathfrak{g},[\cdot,\cdot],\alpha)$ be a multiplicative Hom-Lie superalgebra. Define the following $n$-linear map $[\cdot,\dots,\cdot]_n:\mathfrak g^{\otimes n}\to \mathfrak g \label{crochet_n}$:
 \begin{equation}\label{bracketnuplet}
\Big[x_1,\dots,x_n\Big]_n=\Big[\big[[\dots[x_1,x_2],\alpha(x_3)],\alpha^2(x_4)\big]\dots\alpha^{n-3}(x_{n-1})],\alpha^{n-2}(x_{n})\Big].
\end{equation}
For n=2, $[x_1,x_2]_2=[x_1,x_2]$ and for $n\geq 3$, $[x_1,\dots,x_n]_n=[[x_1,\dots,x_{n-1}]_{n-1},\alpha^{n-2}(x_n)]$.

\begin{thm}\label{n-upletfrom Lie}
Let $(\mathfrak{g},[\cdot,\cdot], \alpha)$ be a multiplicative  Hom-Lie superalgebra. Then
$$\mathfrak{g}_n=(\mathfrak{g},[\cdot ,\dots,\cdot ]_n, \alpha^{n-1})$$
is a multiplicative $n$-ary Hom-Nambu superalgebra.
\end{thm}
To prove this theorem we need the following lemma.
\begin{lem}\label{adj2}
Let $(\mathfrak{g},[\cdot ,\cdot ], \alpha)$ be a multiplicative Hom-Lie superalgebra, and $\text{ad}^2$ the adjoint map defined by
$\text{ad}_x^2(y)=\text{ad}_x(y)=[x,y].$ Then, we have
$$\text{ad}_{\alpha^{n-1}(x)}^2[y_1,\dots,y_n]_n=\displaystyle\sum_{k=1}^n
(-1)^{|x||Y|^{k-1}}[\alpha(y_1),\dots,\alpha(y_{k-1}),\text{ad}_x^2(y_k),\alpha(y_{k+1}),\dots,\alpha(y_n)]_n,$$
where $x\in\mathcal{H}(\mathfrak{g}), y\in \mathcal{H}(\mathfrak g)$ and $(y_1,\dots,y_n)\in\mathcal{H}(\mathfrak{g})^n$.
\end{lem}
\begin{proof}
  For $n=2$, using the super-Hom-Jacobi  identity we have
  \begin{align*}
\text{ad}_{\alpha(x)}^2[y,z]&=[\alpha(x),[y,z]]=[[x,y],\alpha(z)]+(-1)^{|x||y|}[\alpha(y),[x,z]]\\&=
[\text{ad}_x^2(y),\alpha(z)]+(-1)^{|x||y|}[\alpha(y),\text{ad}_x^2(z)].
  \end{align*}
 Assume that the property is true up to order $n$, that is $$\text{ad}_{\alpha^{n-1}(X)}^2[y_1,\dots,y_n]_n=\displaystyle\sum_{k=1}^n
(-1)^{|X||Y|^{k-1}}[\alpha(y_1),\dots,\alpha(y_{k-1}),\text{ad}_X^2(y_k),\alpha(y_{k+1}),\dots,\alpha(y_n)]_n.$$
Let $x\in\mathcal{H}(\mathfrak{g})$ and $(y_1,\dots,y_{n+1})\in\mathcal{H}(\mathfrak{g})^{n+1}$, we have
\begin{align*}
&\text{ad}^2_{\alpha^n(x)}[y_1,\dots,y_{n+1}]=\text{ad}^2_{\alpha^n(x)}[[y_1,\dots,y_n]_n,\alpha^{n-1}(y_{n+1})]_2\\
&=\Big[\text{ad}^2_{\alpha^{n-1}(x)}[y_1,\dots,y_n]_n,\alpha^n(y_{n+1})\Big]_2+
(-1)^{|x||Y|}\Big[[\alpha(y_1),\dots,\alpha(y_n)]_n,\text{ad}^2_{\alpha^{n-1}(x)}(\alpha^{n-1}(y_{n+1}))\Big]_2\\
&=\displaystyle\sum_{k=1}^n\Big[[\alpha(y_1),\dots,\alpha(y_{k-1}),\text{ad}^2_x(y_k),\alpha(y_{k+1}),\dots,\alpha(y_n)]_n,\alpha^n(y_{n+1})\Big]\\
&+(-1)^{|x||Y|^{k-1}}\Big[[\alpha(y_1),\dots,\alpha(y_n)]_n,\alpha^{n-1}(\text{ad}^2_x(y_{n+1}))\Big]_2\\
&=\displaystyle\sum_{k=1}^n(-1)^{|x||Y|^{k-1}}\Big[\alpha(y_1),\dots,\alpha(y_{k-1}),\text{ad}^2_x(y_k),\alpha(y_{k+1}),\dots,\alpha(y_n),\alpha(y_{n+1})\Big]_{n+1}\\
&+\Big[\alpha(y_1),\dots,\alpha(y_n),\text{ad}^2_x(y_{n+1})\Big]_{n+1}\\
&=\displaystyle\sum_{k=1}^{n+1}(-1)^{|x||Y|^{k-1}}\Big[\alpha(y_1),\dots,\alpha(y_{k-1}),\text{ad}^2_x(y_k),\alpha(y_{k+1}),\dots,\alpha(y_{n+1})\Big]_{n+1}.
\end{align*}
The lemma is proved.
\end{proof}

\begin{proof}(\textbf{Proof of Theorem \ref{n-upletfrom Lie}})
Let $X=(x_1,\dots,x_{n-1})\in\mathcal{H}(\mathfrak{g})^{n-1}$ and $Y=(y_1,\dots,y_n)\in\mathcal{H}(\mathfrak{g})^n$. Using Lemma \ref{adj2}, we have
\begin{align*}
&\Big[\alpha^{n-1}(x_1),\dots,\alpha^{n-1}(x_{n-1}),[y_1,\dots,y_n]_n\Big]_n\\=&
\Big[[\alpha^{n-1}(x_1),\dots,\alpha^{n-1}(x_{n-1})]_{n-1},[\alpha^{n-2}(y_1),\dots,\alpha^{n-2}(y_n)]_n\Big]_2\\=&
\text{ad}^2_{\alpha^{n-1}[x_1,\dots,x_{n-1}]_{n-1}}([\alpha^{n-2}(y_1),\dots,\alpha^{n-2}(y_n)]_n)\\=&
\displaystyle\sum_{k=1}^n(-1)^{|X||Y|^{k-1}}\Big[\alpha^{n-1}(y_1),\dots,\text{ad}^2_{[x_1,\dots,x_{n-1}]_{n-1}}(\alpha^{n-2}(y_k)),\dots,\alpha^{n-1}(y_n)\Big]_n\\=&
\displaystyle\sum_{k=1}^n(-1)^{|X||Y|^{k-1}}\Big[\alpha^{n-1}(y_1),\dots,[[x_1,\dots,x_{n-1}]_{n-1},\alpha^{n-2}(y_k)]_2,\dots,\alpha^{n-1}(y_n)\Big]_n\\=&
\displaystyle\sum_{k=1}^n(-1)^{|X||Y|^{k-1}}\Big[\alpha^{n-1}(y_1),\dots,[x_1,\dots,x_{n-1},y_k]_n,\dots,\alpha^{n-1}(y_n)\Big]_n.
\end{align*}
\end{proof}
\begin{example}Consider the $2$-dimensional multiplicative Hom-Lie superalgebras $\mathfrak{g}^3_{1,1}$ and $\mathfrak{g}^4_{1,1}$ given in Theorem \ref{classification2dimension}. We can construct a multiplicative $n$-ary Hom-Nambu superalgebras structures on $\mathfrak{g}^3_{1,1}$ and $\mathfrak{g}^4_{1,1}$ given respectively by:
$$[e_1,e_0,\dots,e_0]^3_n=(-1)^{n-1}e_1\ \
\textrm{and}\ \ [e_1,e_0,\dots,e_0]^4_n=-(-a)^{n-2}e_1.$$
The other brackets are zero.
\end{example}
\begin{example}[\cite{WangZhangWei2016:claslowdimmulthomliesuperalg}] Consider a $3$-dimensional graded linear space $L=L_0\oplus L_1$, where $L_0$ is generated by $e_1$ and   $L_1$ is generated by $e_2 ,e_3$. Define an even linear map $\alpha:L\to L$ by
$$\alpha(e_1)=ae_1,\ \ \alpha(e_2)= ae_2 \ \ \alpha(e_3)= e_3$$ and an even super-skewsymmetric bilinear map $[\cdot,\cdot]:L\times L\to L$ given by
$$[e_1,e_2]=[e_2,e_2]=[e_3,e_3]=0,\ \ [e_1,e_3]=be_2,\ \ [e_3,e_3]=ce_1,$$
where $a,b,c$ are parameters and $a\neq0$.
Then $(L,[\cdot,\cdot],\alpha)$ is a multiplicative Hom-Lie superalgebra.
Therefore, using Theorem \ref{n-upletfrom Lie}, we can construct a multiplicative $3$-ary Hom-Nambu superalgebra $(L,[\cdot,\cdot,\cdot]_3,\alpha^2)$, where the ternary bracket $[\cdot,\cdot,\cdot]_3$ is given by $[e_1,e_3,e_3]= bce_1$ and $[e_2,e_3,e_3]= bce_2$. We can also construct a multiplicative $n$-ary Hom-Nambu superalgebra $(L,[\cdot,\dots,\cdot]_n,\alpha^{n-1})$, where the $n$-ary bracket $[\cdot,\dots,\cdot]_n$ is given by:\\
$\bullet$ If $n=4p$, then $[e_1,e_3,e_3,\dots,e_3]= b^{p}c^{p+1}e_1$ and $[e_2,e_3,e_3,\dots,e_3]= b^{p+1}c^{p}e_2$.
\\
$\bullet$ If $n=4p+1$, $[e_1,e_3,e_3,\dots,e_3]= b^{p+1}c^{p+1}e_2$ and $[e_2,e_3,e_3,\dots,e_3]= b^{p+1}c^{p+1}e_1$.
\\
$\bullet$ If $n=4p+2$, then $[e_1,e_3,e_3,\dots,e_3]= b^{p+1}c^{p+2}e_1$ and $[e_2,e_3,e_3,\dots,e_3]= b^{p+2}c^{p+1}e_2$.\\
$\bullet$ If $n=4p+3$, then $[e_1,e_3,e_3,\dots,e_3]= b^{p+2}c^{p+2}e_2$ and $[e_2,e_3,e_3,\dots,e_3]= b^{p+2}c^{p+2}e_1$.

\end{example}
\begin{proposition}
  Let $(\mathfrak{g},[\cdot ,\cdot ], \alpha)$ be a multiplicative Hom-Lie superalgebra and $\mathfrak{D}:\mathfrak{g}\to\mathfrak{g}$
  an $\alpha^k$-derivation of $\mathfrak{g}$ for an integer $k$. Then $\mathfrak{D}$ is an $\alpha^k$-derivation of $\mathfrak{g}_n$.
\end{proposition}
\begin{proof}
We use the mathematical induction. For $n=3$, given $x,y,z\in\mathcal{H}(\mathfrak{g})$, we have
  \begin{align*}
  &\mathfrak{D}([x,y,z])=\mathfrak{D}([[x,y],\alpha(z)])\\
=&[\mathfrak{D}([x,y]),\alpha^{k+1}(z)]+(-1)^{|\mathfrak{D}||[x,y]|}[[\alpha^k(x),\alpha^k(y)],\mathfrak{D}(\alpha(z))]\\
=&[[\mathfrak{D}(x),\alpha^k(y)],\alpha^{k+1}(z)]+(-1)^{|\mathfrak{D}||x|}[[\alpha^k(x),\mathfrak{D}(y)],\alpha^{k+1}(z)] \\
 & \hspace{5cm} +(-1)^{|\mathfrak{D}|(|x|+|y|)}[[\alpha^k(x),\alpha^k(y)],\alpha(\mathfrak{D}(z))]\\
=&[\mathfrak{D}(x),\alpha^k(y),\alpha^k(z)]+(-1)^{|\mathfrak{D}||x|}[\alpha^k(x),\mathfrak{D}(y),\alpha^k(z)]
  +(-1)^{|\mathfrak{D}|(|x|+|y|)}[\alpha^k(x),\alpha^k(y),\mathfrak{D}(z)].
  \end{align*}
  Now, suppose that the property is true to order $n-1$:
  $$\mathfrak{D}([x_1,\dots,x_{n-1}]_{n-1})=
  \displaystyle\sum_{i=1}^n(-1)^{|\mathfrak{D}||X|^{i-1}}[\alpha^k(x_1),\dots,D(x_k),\dots,\alpha^k(x_{n-1})]_{n-1}.$$
 If $(x_1,\dots,x_n)\in\mathfrak{g}^n$, then
  \begin{align*}
\mathfrak{D}([x_1,\dots,x_n]_n)=&\mathfrak{D}(\Big[[x_1,\dots,x_{n-1}]_{n-1},\alpha^{n-2}(x_n)\Big])\\
=&\Big[\mathfrak{D}([x_1,\dots,x_{n-1}]_{n-1}),\alpha^{n+k-2}(x_n)\Big]\\
&+(-1)^{|\mathfrak{D}||[x_1,\dots,x_{n-1}]_{n-1}|}\Big[[\alpha^k(x_1),\dots,\alpha^k(x_{n-1})]_{n-1},\mathfrak{D}(\alpha^{n-2}(x_n))\Big]\\
=&\Big[\mathfrak{D}([x_1,\dots,x_{n-1}]_{n-1}),\alpha^{n-2}(\alpha^k(x_n))\Big]
\\&+(-1)^{|\mathfrak{D}||X|^{n-1}}\Big[[\alpha^k(x_1),\dots,\alpha^k(x_{n-1})]^{n-1},\alpha^{n-2}(\mathfrak{D}(x_n))\Big]\\
=&\displaystyle\sum_{i=1}^{n-1}(-1)^{|\mathfrak{D}||X|^{i-1}}
\Big[[\alpha^k(x_1),\dots,\mathfrak{D}(x_i),\dots,\alpha^k(x_{n-1})]_{n-1},\alpha^{n-2}(\alpha^k(x_n))\Big]\\
&+(-1)^{|\mathfrak{D}||X|^{n-1}}[\alpha^k(x_1),\dots,\alpha^k(x_{n-1}),\mathfrak{D}(x_n)]_n\\
=& \displaystyle\sum_{i=1}^{n-1}(-1)^{|\mathfrak{D}||X|^{i-1}}[\alpha^k(x_1),\dots,\mathfrak{D}(x_i),\dots,\alpha^k(x_{n-1}),\alpha^k(x_n)]_n\\
&+(-1)^{|\mathfrak{D}||X|^{n-1}}[\alpha^k(x_1),\dots,\alpha^k(x_{n-1}),\mathfrak{D}(x_n)]_n\\
=&\displaystyle\sum_{i=1}^n(-1)^{|\mathfrak{D}||X|^{i-1}}[\alpha^k(x_1),\dots,\mathfrak{D}(x_i),\dots,\alpha^k(x_{n-1}),\alpha^k(x_n)]_n,
  \end{align*}
which completes the proof.
\end{proof}
\begin{proposition}Let $(\mathfrak{g},[\cdot ,\cdot ], \alpha)$ be a multiplicative Hom-Lie superalgebra.
For endomorphisms $\mathfrak{D}$,$\mathfrak{D}'$,$\dots$,$\mathfrak{D}^{(n-1)}$ of $\mathfrak{g}$ such that
$\mathfrak{D}^{(i)} $ is an $\alpha^k$-quasiderivation with associated endomorphism
$\mathfrak{D}^{(i+1)} $ for  $0\leq i\leq n-2$, the $(n + 1)$-tuple
$(\mathfrak{D},\mathfrak{D},\mathfrak{D}',\mathfrak{D}'', \dots ,\mathfrak{D}^{(n-1)})$ is an $(n + 1)$-ary
$\alpha^k$-derivation of $\mathfrak{g}_{n}$.
\end{proposition}
\begin{proof}
Let $x_1,\dots,x_n\in \mathfrak{g}$, then
\begin{align*}
\mathfrak{D}^{(n-1)}([x_1,\dots,x_n]_n)&=\mathfrak{D}^{(n-1)}([[x_1,\dots,x_{n-1}]_{n-1},x_n])\\
&=[\mathfrak{D}^{(n-2)}([x_1,\dots,x_{n-1}]_{n-1}),\alpha^k(x_n)]\\
&+(-1)^{|\mathfrak{D}^{(n-2)}||X|^{n-1}}[[\alpha^k(x_1),\dots,\alpha^k(x_{n-1})]_{n-1},\mathfrak{D}^{(n-2)}(x_n)]\\
&=[[\mathfrak{D}^{(n-3)}([x_1,\dots,x_{n-2}]),\alpha^k(x_n)],\alpha^k(x_n)]\\
&+(-1)^{|\mathfrak{D}^{(n-3)}||X|^{n-2}}\Big[\Big[[\alpha^k(x_1),\dots,\alpha^k(x_{n-2})],\mathfrak{D}^{(n-3)}(x_{n-1})\Big],\alpha^k(x_n)\Big]\\
&+(-1)^{|\mathfrak{D}^{(n-2)}||X|^{n-1}}\Big[[\alpha^k(x_1),\dots,\alpha^k(x_{n-1})]_{n-1},\mathfrak{D}^{(n-2)}(x_n)\Big]\\
&\vdots\\
&=[\mathfrak{D}(x_1),\alpha^k(x_2),\dots,\alpha^k(x_n)]_n+(-1)^{|\mathfrak{D}||x_1|}[\alpha^k(x_1),\mathfrak{D}(x_2),\dots,\alpha^k(x_n)]_n\\
&+(-1)^{|\mathfrak{D}|(|x_1|+|x_2|)}[\alpha^k(x_1),\alpha^k(x_2),\mathfrak{D}'(x_3),\dots,\alpha^k(x_n)]_n\\
&+\dots+(-1)^{|\mathfrak{D}^{(n-2)}||X|^{n-1}}[\alpha^k(x_1),\dots,\alpha^{k}(x_{n-1}),\mathfrak{D}^{(n-2)}(x_n)]_n).
\end{align*}
Therefore  $(n + 1)$-tuple $(\mathfrak{D},\mathfrak{D},\mathfrak{D}',\mathfrak{D}'', \dots ,\mathfrak{D}^{(n-1)})$ is a generalized
$\alpha^k$-derivation of $\mathfrak{g}_{n}$.
\end{proof}

\section{Rota-Baxter $n$-ary Hom-Nambu superalgebras} \label{sec:rotabaxternaryhomNambusuperalg}

In this section, we introduce the notion of Rota-Baxter operators of Hom-Nambu superalgebras and $3$-Hom-pre-Lie algebras. Then we introduce the notion of a $3$-Hom-pre-Lie superalgebra which is closely related to Rota-Baxter operators of weight $0$.

\subsection{Rota-Baxter Operator on $n$-ary Hom-Nambu superalgebras}

Let $(A, \cdot,\alpha)$ be a $\mathbb{K}$-super-linear space with an even  binary operation $\cdot$ and linear map $\alpha:A\to A$ and let $\lambda\in \mathbb{K}$.
If an even linear map $R: A\to A$ satisfies, for all  $x, y \in A$,
\begin{equation}\label{RBaxterBinary}R\alpha=\alpha R,\ \
R(x)\cdot R(y)=R(R(x)\cdot y+x\cdot R(y)+\lambda x\cdot y) ,
\end{equation}
then $R$ is called a  Rota-Baxter operator of weight $\lambda$ on Hom-superalgebra $(A, \cdot, \alpha)$.

We generalize the concepts of a Rota-Baxter operator  to $n$-ary Hom-Nambu superalgebras.
\begin{definition}
 Let $\lambda\in \mathbb{K}$ and
an  $n$-ary Hom-Nambu superalgebra  $(\mathcal{N}, [\cdot , \dots, \cdot],\alpha)$. A Rota-Baxter operator of weight $\lambda$ on $(A,[\cdot , \dots, \cdot],\alpha)$ is an even linear map $R: \mathcal{N}\to \mathcal{N}$ such that $R\alpha=\alpha R$ satisfying
\begin{equation}
[ R(x_1), \dots, R(x_n)]
=R\big( \sum\limits_{\emptyset \neq I\subseteq [n]}\lambda^{|I|-1} [ \hat{R}(x_1), \dots, \hat{R}(x_i), \dots, \hat{R}(x_{n})]\big),
\end{equation}
where
$\hat{R}(x_i):=\hat{R}_I(x_i):=\left\{\begin{array}{ll} x_i, & i\in I, \\ R(x_i), & i\not\in I \end{array}\right. \text{ for all } x_1,\dots, x_n\in \mathcal{N}.
$
In particular, a  Rota-Baxter operator of weight $\lambda$  of  ternary Hom-Nambu superalgebra $(\mathcal{N},[\cdot , \dots, \cdot],\alpha)$ is  an even linear map
$R: \mathcal{N}\to \mathcal{N}$ commuting with $\alpha$ such that
\begin{eqnarray*}
[ R(x_1),R(x_2),R(x_3)]
&=& R\Big(
[ R(x_1),R(x_2),x_3] +[ R(x_1),x_2,R(x_3)] +[ x_1,R(x_2),R(x_3)]  \\
&+&\lambda [ R(x_1),x_2,x_3]
+\lambda [ x_1,R(x_2),x_3]
+\lambda [ x_1,x_2,R(x_3)]\\
&+&\lambda^2 [x_1,x_2,x_3]\Big).
\end{eqnarray*}

\end{definition}
\begin{proposition}
  Let $(\mathcal{N}, [\cdot , \dots, \cdot],\alpha)$  be a $n$-ary Hom-Nambu superalgebra over a field $\mathbb{K}$. An invertible even linear mapping $R : \mathcal{N}\to \mathcal{N}$
is a Rota-Baxter operator of weight $0$ on $\mathcal{N}$ if and only if $R^{-1}$ is an even derivation
on $\mathcal{N}$.
\end{proposition}
\begin{proof}
$R$ is an even invertible Rota-Baxter operator of weight $0$ on $\mathcal{N}$ if and only if
\begin{align*}
\forall x_1,\dots,x_n\in A: \, [R(x_1),\dots,R(x_n)]=R\Big(\displaystyle\sum_{i=1}^n[R(x_1),\dots,x_i,\dots,R(x_n)]\Big).
\end{align*}
For $X_k=R(x_k),\;k\in\{1,\dots,n\}$:
$[X_1,\dots,X_n]=R\Big(\displaystyle\sum_{i=1}^n[X_1,\dots,R^{-1}(X_i),\dots,X_n]\Big).$
Hence,
$R^{-1}([X_1,\dots,X_n])=\displaystyle\sum_{i=1}^n[X_1,\dots,R^{-1}(X_i),\dots,X_n].$
Thus $R^{-1}$ is an even derivation on $A$.
\end{proof}
\begin{proposition}
   Let $R$ be a Rota-Baxter of weight $0$ of Hom-Lie superalgebra $(\mathfrak{g},[\cdot,\cdot],\alpha)$ and $\phi\in\wedge^{n-2}\mathfrak g^*$ an even $(n-2)$-cochaine satisfying the conditions \eqref{NHomLieProduct} and \eqref{NHomLieProduct1}. Then $R$ is a Rota-Baxter operator of weight $0$ on the $n$-ary Hom-Nambu superalgebra  $(\mathfrak{g},[\cdot,\dots,\cdot]_\phi,\alpha)$  defined in \eqref{crochet_n}
 if and only if $R$ satisfies
\begin{align}
&\forall x_1,\dots,x_n\in\mathcal{H}(\mathfrak{g}): \nonumber \\
&\sum_{k<l}^{n}\Big(\sum_{i\neq k,l}^{n} \phi\big( R(x_1), \dots ,\widehat{R(x_k)}, \dots,\widehat{R(x_l)}, \dots, R(x_{n})\big)\Big)[R(x_k),R(x_l)]\in \ker(R).
\label{RotaBaxterCondition}
\end{align}
\end{proposition}
\begin{proof}
For $X=(x_1,\dots,x_n)\in\mathcal{H}(\mathfrak{g})^{\otimes n}$,
\begin{align*}
&[R(x_1),\dots,R(x_n)]_{\phi}= \\=&\displaystyle\sum_{k<l}(-1)^{|k|+|l|+1}(-1)^{\gamma_{kl}^X}
\phi(R(x_1),\dots,\widehat{R(x_k)},\dots,\widehat{R(x_l)},\dots,R(x_n))[R(x_k),R(x_l)]\\
=&\displaystyle\sum_{k<l}(-1)^{|k|+|l|+1}(-1)^{\gamma_{kl}^X}
\phi\Big(R(x_1),\dots,\widehat{R(x_k)},\dots,\widehat{R(x_l)},\dots,R(x_n)\Big)([R(x_k),x_l]+[x_k,R(x_l)])
\end{align*}
On the other hand,
\begin{align*}
&R\Big(\displaystyle\sum_{i=1}^n[R(x_1),\dots,x_i,\dots,R(x_n)]_{\phi}\Big)=\\=&
R\Big(\displaystyle\sum_{i=1}^n\displaystyle\sum_{k<l;k,l\neq i}(-1)^{k+l+1}(-1)^{\gamma_{kl}^X}
\phi(R(x_1),\dots,\widehat{R(x_k)},\dots,x_i,\dots,\widehat{R(x_l)},\dots,R(x_n)) \Big. \\
& \hspace{12cm} \Big. [R(x_k),R(x_l)]\Big)\\
+&R\Big(\displaystyle\sum_{i=1}^n\displaystyle\sum_{k<i}(-1)^{k+i+1}(-1)^{\gamma_{ki}^X}
\phi(R(x_1),\dots,\widehat{R(x_k)},\dots,\widehat{x_i},\dots,R(x_n))[R(x_k),x_i]\Big)\\
+&R\Big(\displaystyle\sum_{i=1}^n\displaystyle\sum_{i<k}(-1)^{k+i+1}(-1)^{\gamma_{ik}^X}
\phi(R(x_1),\dots,\widehat{x_i},\dots,\widehat{R(x_k)},\dots,R(x_n))[x_i,R(x_k)]\Big)
\end{align*}
It is easy to see that
\begin{align*}
& R\Big(\displaystyle\sum_{i=1}^n[R(x_1),\dots,x_i,\dots,R(x_n)]_{\phi}\big)-[R(x_1),\dots,R(x_n)]_{\phi}=\\
& R\Big(\displaystyle\sum_{i=1}^n\displaystyle\sum_{\substack{k<l \\ k,l\neq i}}(-1)^{k+l+1}(-1)^{\gamma_{kl}^X}
\phi(R(x_1),\dots,\widehat{R(x_k)},\dots,x_i,\dots,\widehat{R(x_l)},\dots,R(x_n))[R(x_k),R(x_l)]\Big)
\end{align*}
Then,  $R$ is a Rota-Baxter operator on the $n$-ary Hom-Nambu superalgebra  $(\mathfrak{g},[\cdot,\dots,\cdot]_\phi,\alpha)$  defined in \eqref{crochet_n}
 if and only if
\begin{multline*}
R\Big(\displaystyle\sum_{i=1}^n\displaystyle\sum_{k<l;k,l\neq i}(-1)^{k+l+1}(-1)^{\gamma_{kl}^X}
\phi(R(x_1),\dots,\widehat{R(x_k)},\dots,x_i,\dots,\widehat{R(x_l)},\dots,R(x_n))\Big. \\
\Big. [R(x_k),R(x_l)]\Big)=0,
\end{multline*}
which gives
\begin{multline*} \displaystyle\sum_{k<l}^{n}(-1)^{k+l+1}(-1)^{\gamma_{kl}^X}\Big(\sum_{i\neq k,l}^{n} \phi\big( R(x_1), \dots ,\widehat{R(x_k)}, \dots, x_i, \dots,\widehat{R(x_l)}, \dots, R(x_{n})\big)\Big. \\
[R(x_k),R(x_l)]\in \ker(R).
\end{multline*}
\end{proof}
\begin{proposition}
   A Rota-Baxter $R$ of weight $0$ of Hom-Lie superalgebra $(\mathfrak{g},[\cdot,\cdot],\alpha)$ is a Rota-Baxter operator on the associated $n$-ary Hom-Nambu algebra    $(\mathfrak{g},[\cdot,\dots,\cdot]_n,\alpha^{n-2})$  defined in \eqref{bracketnuplet}.
\end{proposition}
\begin{proof}It easy to show  that $\alpha^pR=R\alpha^p$ for any integer $p$.
We use the mathematical induction on the integer $n\geq3$:\\
i) For $n=3$:
Let $x,y,z\in\mathcal{H}(\mathfrak{g})$, we have:
\begin{align*}
& [R(x),R(y),R(z)]_3 = [[R(x),R(y)],\alpha(R(z))] = \\
& =[R([R(x),y]),R(\alpha(z))]+[R([x,R(y)]),R(\alpha(z))]\\
& =R([R([R(x),y]),\alpha(z)])+R([[R(x),y],R(\alpha(z))]) \\
& +R([R([x,R(y)]),\alpha(z)])+R([[x,R(y)],R(\alpha(z))])\\
& =R([R([R(x),y]),\alpha(z)])+R([R([x,R(y)]),\alpha(z)]) \\
& +R([R(x),y,R(z)]_3)+R([x,R(y),R(\alpha(z))]_3)\\
& =R([R([R(x),y])+R([x,R(y)]),\alpha(z)])\\
& +R([R(x),y,R(z)]_3)+R([x,R(y),R(z)]_3)\\
& =R([R(x),R(y),z]_3)+R([R(x),y,R(z)]_3)+R([x,R(y),R(z)]_3)\\
& =[R(x),R(y),R(z)]_3
\end{align*}
ii) Assume the property is true to order $n>3$, that is:
\begin{align*}
& \forall (x_1,\dots,x_{n-1})\in\mathcal{H}(\mathfrak{g})^{\otimes n-1}: \\
& [R(x_1),\dots,R(x_{n-1})]_{n-1}=R\Big(\displaystyle\sum_{i=1}^{n-1}[R(x_1),\dots,x_i,\dots,R(x_{n-1})]_{n-1}\Big).
\end{align*}
For $(x_1,\dots,x_n)\in\mathcal{H}(\mathfrak{g})^{\otimes n}$,
\begin{align*}
&[R(x_1),\dots,R(x_n)]_n=[[R(x_1),\dots,R(x_{n-1})]_{n-1},\alpha^{n-2}(R(x_n))]\\
&=\displaystyle\sum_{i=1}^{n-1}\Big[R([R(x_1),\dots,x_i,\dots,R(x_{n-1})]_{n-1}),R(\alpha^{n-2}(x_n))\Big]\\
&=R\Big(\displaystyle\sum_{i=1}^{n-1}\Big[[R(x_1),\dots,x_i,\dots,R(x_{n-1})]_{n-1},\alpha^{n-2}(R(x_n))\Big]\Big)\\
&+R\Big(\displaystyle\sum_{i=1}^{n-1}\Big[R([R(x_1),\dots,x_i,\dots,R(x_{n-1})]_{n-1}),\alpha^{n-2}(x_n)\Big]\Big)\\
&=R\Big(\displaystyle\sum_{i=1}^{n-1}[R(x_1),\dots,x_i,\dots,R(x_{n-1}),R(x_n)]_n\Big)\\&+
R\Big(\Big[\displaystyle\sum_{i=1}^{n-1}R([R(x_1),\dots,x_i,\dots,R(x_{n-1})]_{n-1}),\alpha^{n-2}(x_n)\Big]\Big)\\
&=R\Big(\displaystyle\sum_{i=1}^{n-1}[R(x_1),\dots,x_i,\dots,R(x_{n-1}),R(x_n)]_n\Big)+
R\big([[R(x_1),\dots,R(x_{n-1})]_{n-1},\alpha^{n-2}(x_n)]\Big)\\
&=R\Big(\displaystyle\sum_{i=1}^{n-1}[R(x_1),\dots,x_i,\dots,R(x_{n-1}),R(x_n)]_n\Big)+
R\big([R(x_1),\dots,R(x_{n-1}),x_n]_n\Big)\\
&=R\Big(\displaystyle\sum_{i=1}^n[R(x_1),\dots,x_i,R(x_n)]_n\Big).
\end{align*}
The theorem is proved.
\end{proof}
\subsection{$3$-Hom-pre-Lie superalgebras}
\

\vskip0.2cm

In this subsection,  we generalize the notion of a $3$-Hom-pre-Lie
algebra introduced in \cite{BaiGuoSheng:BialgYangBaxtereqManintri3Liealg} to the super case, which is closely related to Rota Baxter operators. In particular,
there is a construction of  $3$-Hom-pre-Lie superalgebras obtained from
$3$-Hom-Lie superalgebras.

\begin{definition}
A triple $(A,\{\cdot,\cdot,\cdot\},\alpha)$, consisting of a linear super-space $A$ and two even linear maps $\{\cdot,\cdot,\cdot\}:A\otimes A\otimes A\to A$ and $\alpha:A\to A$, is called a {\it $3$-Hom-pre-Lie superalgebra} if the following identities hold:
\begin{align}
\{x,y,z\}&=-(-1)^{|x||y|}\{y,x,z\},\label{eq:d1}\\
\nonumber \{\alpha(x_1),\alpha(x_2),\{x_3,x_4,x_5\}\} &=\{[x_1,x_2,x_3]_C,\alpha(x_4),\alpha(x_5)\} \\
\nonumber &+(-1)^{|x_3|(|x_1|+|x_2|)}\{\alpha(x_3),[x_1,x_2,x_4]_C,\alpha(x_5)\}\\
&+(-1)^{(|x_1|+|x_2|)(|x_3|+|x_4|)}\{\alpha(x_3),\alpha(x_4),\{x_1,x_2,x_5\}\},\label{eq:d2}\\
\nonumber \{ [x_1,x_2,x_3]_C,\alpha(x_4), \alpha(x_5)\} &=\{\alpha(x_1),\alpha(x_2),\{ x_3,x_4, x_5\}\} \\
\nonumber &+(-1)^{|x_1|(|x_2|+|x_3|)}\{\alpha(x_2),\alpha(x_3),\{ x_1,x_4,x_5\}\}\\
&+(-1)^{|x_3|(|x_1|+|x_2|)}\{\alpha(x_3),\alpha(x_1),\{ x_2,x_4, x_5\}\},\label{eq:d3}
\end{align}
 where $x,y,z, x_i\in\mathcal{H}( A), 1\leq i\leq 5$ and $[\cdot,\cdot,\cdot]_C$ is called $3$-supercommutator and  defined by
\begin{align}
& \forall  x,y,z\in \mathcal{H}(A): \nonumber \\ %
& [x,y,z]_C=\{x,y,z\}+(-1)^{|x|(|y|+|z|)}\{y,z,x\}+(-1)^{|z|(|x|+|y|)}\{z,x,y\}.\label{eq:3cc}
\end{align}
\end{definition}

\begin{proposition}\label{PreLieSupAlgToLieSupAlg}
Let $(A,\{\cdot,\cdot,\cdot\},\alpha)$ be a $3$-Hom-pre-Lie superalgebra. Then the induced $3$-supercommutator in \eqref{eq:3cc} and the linear map $\alpha$ define a $3$-Hom-Lie superalgebra on $A$.
\end{proposition}

\begin{proof} By \eqref{eq:d1}, the induced $3$-supercommutator $[\cdot,\cdot,\cdot]_C$ in \eqref{eq:3cc} is super-skew-symmetric. For $x_1,x_2,x_3,x_4,x_5\in \mathcal{H}(A)$,
\begin{align*}
&[\alpha(x_1),\alpha(x_2),[x_3,x_4,x_5]_C]_C-[[x_1,x_2,x_3]_C,\alpha(x_4),\alpha(x_5)]_C\\
&\hspace{3cm} -(-1)^{|x_3|(|x_1|+|x_2|)}[\alpha(x_3),[x_1,x_2,x_4]_C,\alpha(x_5)]_C\\
&\hspace{5cm} -(-1)^{(|x_1|+|x_2|)(|x_3|+|x_4|)}[\alpha(x_3),\alpha(x_4),[ x_1,x_2, x_5]_C]_C\\
&=\{\alpha(x_1),\alpha(x_2),\{x_3,x_4,x_5\}\}+(-1)^{|x_3|(|x_4|+|x_5|)}\{\alpha(x_1),\alpha(x_2),\{x_4,x_5,x_3\}\}\\
&+(-1)^{|x_5|(|x_3|+|x_4|)}\{\alpha(x_1),\alpha(x_2),\{x_5,x_3,x_4\}\}\\
&+(-1)^{|x_1|(|x_2|+|x_3|+|x_4|+|x_5|)}\{\alpha(x_2),[x_3,x_4,x_5]_C, \alpha(x_1)\}\\
&+(-1)^{(|x_1|+|x_2|)(|x_3|+|x_4|+|x_5|)}\{[x_3,x_4,x_5]_C, \alpha(x_1), \alpha(x_2)\}\\
&-\{[x_1,x_2,x_3]_C,\alpha(x_4),\alpha(x_5)\} \\
&-(-1)^{(|x_1|+|x_2|+|x_3|)(|x_4|+|x_5|)}\{\alpha(x_4),\alpha(x_5),\{x_1,x_2,x_3\}\}\\
&-(-1)^{(|x_1|+|x_2|)(|x_3|+|x_4|+|x_5|)}(-1)^{|x_1|(|x_2|+|x_3|)}\{\alpha(x_4),\alpha(x_5),\{x_2,x_3,x_1\}\}\\
&-(-1)^{(|x_1|+|x_2|)(|x_3|+|x_4|+|x_5|)}(-1)^{|x_3|(|x_1|+|x_2|)}\{\alpha(x_4),\alpha(x_5),\{x_3,x_1,x_2\}\}\\
&-(-1)^{|x_5|(|x_1|+|x_2|+|x_3|+|x_4|+|x_5|)}\{\alpha(x_5),[x_1,x_2,x_3]_C,\alpha(x_4)\}\\
&-(-1)^{|x_3|(|x_1|+|x_2|)}\{\alpha(x_3),[x_1,x_2,x_4]_C,\alpha(x_5)\} \\
&-(-1)^{|x_3|(|x_4|+|x_5|)}\{[x_1,x_2,x_4]_C,\alpha(x_5),\alpha(x_3)\}\\
&-(-1){|x_5|(|x_1|+|x_2|+|x_3|+|x_4|)}(-1)^{|x_3|(|x_1|+|x_2|)}\{\alpha(x_5),\alpha(x_3),\{x_1,x_2,x_4\}\}\\
&-(-1)^{|x_5|(|x_1|+|x_2|+|x_3|+|x_4|)}(-1)^{|x_1|(|x_2|+|x_4|)+|x_3|(|x_1|+|x_2|)}\{\alpha(x_5),\alpha(x_3),\{x_2,x_4,x_1\}\}\\
&-(-1)^{|x_5|(|x_1|+|x_2|+|x_3|+|x_4|)}(-1)^{|x_4|(|x_1|+|x_2|)+|x_3|(|x_1|+|x_2|)}\{\alpha(x_5),\alpha(x_3),\{x_4,x_1,x_2\}\}\\
&-(-1)^{|x_3|(|x_4|+|x_5|)+|x_4|(|x_1|+|x_2|)}\{\alpha(x_4),[ x_1,x_2, x_5]_C,\alpha(x_3)\} \\
&-(-1)^{|x_5|(|x_3|+|x_4|)}\{[ x_1,x_2, x_5]_C,\alpha(x_3),\alpha(x_4)\}\\
&-(-1)^{(|x_1|+|x_2|)(|x_3|+|x_4|)}\{\alpha(x_3),\alpha(x_4),\{ x_1,x_2, x_5\}\}\\
&-(-1)^{(|x_1|+|x_2|)(|x_3|+|x_4|+|x_5|)}\{\alpha(x_3),\alpha(x_4),\{ x_5,x_1, x_2\}\}\\
&-(-1)^{|x_1|(|x_2|+|x_5|)}(-1)^{(|x_1|+|x_2|)(|x_3|+|x_4|)}\{\alpha(x_3),\alpha(x_4),\{ x_2,x_5, x_1\}\}=0.
\end{align*}
when applying identities \eqref{eq:d2} and \eqref{eq:d3}. Thus the proof is completed.
\end{proof}

\begin{definition}
Let $(A,\{\cdot,\cdot,\cdot\},\alpha)$ be a $3$-Hom-pre-Lie superalgebra. The $3$-Hom-Lie superalgebra $(A,[\cdot,\cdot,\cdot]_C,\alpha)$
is called the {\it sub-adjacent $3$-Hom-Lie superalgebra} of $(A,\{\cdot,\cdot,\cdot\},\alpha)$ and $(A,\{\cdot,\cdot,\cdot\},\alpha)$ is called a {\it compatible
$3$-Hom-pre-Lie superalgebra} of the $3$-Hom-Lie superalgebra $(A,[\cdot,\cdot,\cdot]_C,\alpha)$.
\end{definition}

New identities of $3$-pre-Hom-Lie superalgebras can be derived from Proposition~\ref{PreLieSupAlgToLieSupAlg}.
\begin{cor}
Let $(A,\{\cdot,\cdot,\cdot\},\alpha)$ be a $3$-Hom-pre-Lie algebra. The following identities hold:
\begin{multline*}
\{ [x_1,x_2,x_3]_C,\alpha(x_4), \alpha(x_5)\}-(-1)^{|x_3||x_4|}\{[x_1,x_2,x_4]_C,\alpha(x_3), \alpha(x_5) \}\\
+(-1)^{|x_2|(|x_3|+|x_4|)}\{[x_1,x_3,x_4]_C,\alpha(x_2), \alpha(x_5)\} \\
-(-1)^{|x_1|(|x_2|+|x_3|+|x_4|)}\{[x_2,x_3,x_4]_C,\alpha(x_1), \alpha(x_5)\}=0,
\end{multline*}
\begin{multline*}
\{\alpha(x_1),\alpha(x_2),\{x_3,x_4, x_5\}\}+(-1)^{(|x_1|+|x_2|)(|x_3|+|x_4|)}\{\alpha(x_3),\alpha(x_4),\{x_1,x_2, x_5\}\} \\
+(-1)^{|x_1|(|x_2|+|x_3|+|x_4|)+|x_3||x_4|}\{\alpha(x_2),\alpha(x_4),\{x_3,x_1, x_5\}\} \\
+(-1)^{|x_3|(|x_1|+|x_2|)}\{\alpha(x_3),\alpha(x_1),\{x_2,x_4, x_5\}\}\\
+(-1)^{|x_1|(|x_2|+|x_3|)}\{\alpha(x_2),\alpha(x_3),\{x_1,x_4, x_5\}\} \\
+(-1)^{|x_4|(|x_2|+|x_3|)}\{\alpha(x_1),\alpha(x_4),\{x_2,x_3, x_5\}\} =0,
\end{multline*}
for $x_i\in \mathcal{H}(A), 1\leq i\leq 5$.
\end{cor}

\begin{proposition}\label{pro:3preLieT}
Let $(A,[\cdot,\cdot,\cdot],\alpha)$ be a $3$-Hom-Lie superalgebra and  $R:A\to A$ is an operator Rota-baxter of weight $0$. Then there exists a $3$-Hom-pre-Lie superalgebra structure on $A$ given by
\begin{equation}
\{x,y,z\}=[R(x),R(y),z],\quad\forall ~ x,y,z\in \mathcal{H}(A).
\end{equation}
\end{proposition}

\begin{proof} Let $x,y,z\in \mathcal{H}(A)$.
It is obvious that
$$\{x,y,z\}=[R(x),R(y),z]=-(-1)^{|x||y|}[R(y),R(x),z]=-(-1)^{|x||y|}\{y,x,z\}.$$
Furthermore, the following equation holds:
$$
[x,y,z]_C=[R(x),R(y),z]+(-1)^{|z|(|x|+|y|)}[R(z),R(x),y]+(-1)^{|x|(|y|+|z|)}[R(y),R(z),x].
$$
Since $R$ is a Rota-Baxter operator, we have
$$
R([x,y,z]_C)=[R(x),R(y),R(z)].
$$
For $x_1, x_2,x_3,x_4,x_5\in \mathcal{H}(A)$,
\begin{align*}
\{\alpha(x_1),\alpha(x_2),\{x_3,x_4,x_5 \}\}=&[R(\alpha(x_1)),R(\alpha(x_2)),[R(x_3),R(x_4),x_5]]\\
=&[\alpha(R(x_1)),\alpha(R(x_2)),[R(x_3),R(x_4),x_5]];\\
\{[x_1,x_2,x_3]_C,\alpha(x_4),\alpha(x_5)\}=&[R([x_1,x_2,x_3]_C),R(\alpha(x_4)),\alpha(x_5)]\\
=&[[R(x_1),R(x_2),R(x_3)],\alpha(R(x_4)),\alpha(x_5)];\\
\{\alpha(x_3),[x_1,x_2,x_4]_C,x_5\}=&[R(\alpha(x_3)),R([x_1,x_2,x_4]_C),\alpha(x_5)]\\
=&[\alpha(R(x_3)),[R( x_1),R(x_2),R(x_4)],\alpha(x_5)];\\
\{\alpha(x_3),\alpha(x_4),\{x_1,x_2,x_5\}\}=&[R(\alpha(x_3)),R(\alpha(x_4)),[R(x_1),R(x_2),x_5]]\\
=&\alpha(R(x_3)),\alpha(R(x_4)),[R(x_1),R(x_2),x_5]].
\end{align*}
By Condition \eqref{NambuIdentity}, \eqref{eq:d2} holds. On the other hand, we have
\begin{align*}
\{[x_1,x_2,x_3]_C,\alpha(x_4), \alpha(x_5)\}=&[R([x_1,x_2,x_3]_C),R(\alpha(x_4)),\alpha(x_5)]\\
=&[[R(x_1),R(x_2),R(x_3)],\alpha(R(x_4)),\alpha(x_5)];\\
\{\alpha(x_1),\alpha(x_2),\{x_3,x_4, x_5\}\}=&[R(\alpha(x_1)),R(\alpha(x_2)),[R(x_3),R(x_4),x_5]]\\
=&[\alpha(R(x_1)),\alpha(R(x_2)),[R(x_3),R(x_4),x_5];\\
\{\alpha(x_2),\alpha(x_3),\{x_1,x_4, x_5\}\}=&[R(\alpha(x_2)),R(\alpha(x_3)),[R(x_1),R(x_4),x_5]]\\
=&[\alpha(R(x_2)),\alpha(R(x_3)),[R(x_1),R(x_4),x_5];\\
\{\alpha(x_3),\alpha(x_1),\{x_2,x_4, x_5\}\}=&[R(\alpha(x_3)),R(\alpha(x_1)),[R(x_2),R(x_4),x_5]]\\
=&[\alpha(R(x_3)),\alpha(R(x_1)),[R(x_2),R(x_4),x_5]].
\end{align*}
By super-Nambu identity, \eqref{eq:d3} holds. This completes the proof.
\end{proof}

\begin{cor}
With the above conditions,  $(A,[\cdot,\cdot,\cdot]_C,\alpha)$ is a $3$-Hom-Lie
superalgebra as the sub-adjacent $3$-Hom-Lie superalgebra of the $3$-Hom-pre-Lie
superalgebra given in Proposition \ref{pro:3preLieT}, and $R$ is a $3$-Hom-Lie superalgebra morphism from $(A,[\cdot,\cdot,\cdot]_C,\alpha)$ to $(A,[\cdot,\cdot,\cdot],\alpha)$. Furthermore,
$R(A)=\{R(x)|x\in A\}\subset A$ is a $3$-Hom-Lie super-subalgebra of $A$ and there is an induced $3$-Hom-pre-Lie superalgebra structure $(\{\cdot,\cdot,\cdot\}_{R(A)},\alpha)$ on
$R(A)$ given by
\begin{equation}
\{R(x),R(y),R(z)\}_{R(A)}:=R(\{x,y,z\}),\quad\;\forall x,y,z\in \mathcal{H}(A).
\end{equation}
\end{cor}

\begin{proposition}\label{pro:preLieOoper}
Let $(A,[\cdot,\cdot,\cdot],\alpha)$ be a $3$-Hom-Lie superalgebra. Then there exists a compatible $3$-Hom-pre-Lie superalgebra if and only if there exists an invertible Rota-Baxter operator $R$ on $A$.
\end{proposition}

\begin{proof}
Let $R$ be an invertible Rota-Baxter operator of $A$. Then there exists a $3$-Hom-pre-Lie superalgebra structure $(\{x,y,z\},\alpha)$ on $A$
defined by
$$\{x,y,z\}=ad_{(R(x),R(y))}(z),\quad \forall x,y,z\in \mathcal{H}(A).$$
Moreover, there is an induced $3$-Hom-pre-Lie superalgebra structure $(\{\cdot,\cdot,\cdot\}_A,\alpha)$ on $A=R(A)$ given by
$$\{x,y,z\}_A=R\{R^{-1}(x),R^{-1}(y),R^{-1}(z)\}=R(ad_{(x,y)}(R^{-1}(z)))$$
for all $x,y,z\in \mathcal{H}(A)$. Since $R$ is an operator Rota-Baxter on $A$, we have
\begin{align*}
&[x,y,z]=R\Big([x,y,R^{-1}(z)]+[x,R^{-1}(y),z]+[R^{-1}(x),y,z]\Big)\\
&=R\Big(ad_{(x,y)}(R^{-1}(z))+(-1)^{|z|(|x|+|y|)}ad_{(z,x)}(R^{-1}(y))+(-1)^{|x|(|y|+|z|)}ad_{(y,z)}(R^{-1}(x))\Big)\\
&=\{x,y,z\}_A+(-1)^{|z|(|x|+|y|)}\{z,x,y\}_A+(-1)^{|x|(|y|+|z|)}\{y,z,x\}_A
\end{align*}
 Therefore $(A,\{\cdot,\cdot,\cdot\}_A,\alpha)$ is a compatible $3$-Hom-pre-Lie superalgebra of $(A,[\cdot,\cdot,\cdot])$.
\end{proof}

\section*{Acknowlegments} Dr. Sami Mabrouk is grateful to the research environment in Mathematics and
Applied Mathematics MAM, the Division of Applied Mathematics
of the School of Education, Culture and Communication
at M{\"a}lardalen University for hospitality and an excellent and inspiring environment for research and research education and cooperation in Mathematics during his visit in Autumn 2019.

\end{document}